\documentclass[letterpaper,12pt]{amsart}
\usepackage{graphicx}
\usepackage{latexsym}
\usepackage{array,delarray}
\usepackage{amsmath,amssymb, epsfig,verbatim}
\usepackage{color}
\usepackage{hyperref}
\usepackage{tikz}
\usetikzlibrary{matrix,arrows}

\newtheorem{thm}{Theorem}[section]
\newtheorem{lem}[thm]{Lemma}
\newtheorem{prop}[thm]{Proposition}
\newtheorem{cor}[thm]{Corollary}

\newtheorem{defn}[thm]{Definition}
\newtheorem{ex}[thm]{Example}

\newtheorem{rmk}[thm]{Remark}

\newcommand{\rr}{\mathbb{R}}
\newcommand{\cc}{\mathbb{C}}
\newcommand{\kk}{\mathbb{K}}

\newcommand{\calc}{\mathcal{C}}

\newcommand{\cm}{\mathcal{M}}

\newcommand{\rank}{\operatorname{rank}}

\newcommand{\ind}{\mathrel{\mbox{$\perp \kern-5.5pt \perp$}}}
\newcommand{\succdot}{\mathrel{\cdot \kern-0.6em >}}
\newcommand{\bi}{\leftrightarrow}

\newcommand{\newword}{\textbf}

\makeatletter
\renewcommand*\env@matrix[1][*\c@MaxMatrixCols c]{%
 \hskip -\arraycolsep
  \let\@ifnextchar\new@ifnextchar
  \array{#1}}
\makeatother

\begin{document}

\title{Matrix Schubert varieties and\\ Gaussian conditional independence models
}


\author{Alex Fink  \and Jenna Rajchgot \and Seth Sullivant 
}




\maketitle

\begin{abstract}
Matrix Schubert varieties are certain varieties in the affine space of 
square matrices which are determined by specifying rank conditions
on submatrices.  
We study these varieties for generic matrices, symmetric matrices,
and upper triangular matrices in view of two applications to algebraic
statistics: we observe that special conditional independence models for Gaussian
random variables are intersections of matrix Schubert varieties in the
symmetric case. Consequently, we obtain a combinatorial primary decomposition
algorithm for some conditional independence ideals. We also characterize the vanishing ideals of Gaussian graphical
models for generalized Markov chains.

In the course of this investigation, we are led to consider three related stratifications, which come from the 
Schubert stratification of a flag variety. We provide some combinatorial results, including describing the stratifications using the language of rank arrays and enumerating the strata in each case.

\end{abstract}

\tableofcontents

\section{Introduction}

An $m$-dimensional Gaussian random vector $X = (X_{1}, \ldots, X_{m})
 \sim  \mathcal{N}(\mu, \Sigma)$
has its conditional independence structure completely determined
by rank conditions on its covariance matrix $\Sigma$, which is an $m\times m$ symmetric positive definite matrix.
For a subset $A$ of $[m]:=\{1,2,\dots, m\}$, let $X_{A}  = (X_a)_{a \in A}$ be the subvector
of $X$ indexed by $A$.  For $A, B \subseteq [m]$ let $\Sigma_{A,B}$
denote the submatrix of $\Sigma = (\sigma_{i,j})_{i,j \in [m]}$
with row index set $A$ and column index set $B$; that is
$\Sigma_{A,B}  = (\sigma_{a,b})_{a \in A, b \in B}$. 

\begin{prop}\label{prop:GaussianCI}
Let $X \sim \mathcal{N}(\mu, \Sigma)$.  Then for disjoint subsets
$A,B,C \subseteq [m]$, the 
conditional independence statement $X_{A} \ind X_{B} \mid  X_{C}$ (read 
``$X_A$ is independent of $X_B$ given $X_C$'')
holds if and only if
\begin{equation}\label{eq:rank}
{\rm rank} \,  \Sigma_{A \cup C,  B \cup C}  =  \# C.
\end{equation}
\end{prop}
For the purposes of this paper,  Proposition \ref{prop:GaussianCI}
can be taken as the definition of the conditional independence
statement  $X_{A} \ind X_{B} \mid  X_{C}$, since this is the only
fact we will need in our study.  Note that we often
use $A \ind B \mid  C$ to denote the conditional independence statement
 $X_{A} \ind X_{B} \mid  X_{C}$.  A precise definition of conditional
independence structures and the derivation of Proposition
\ref{prop:GaussianCI} can be found in \cite{Drton2009}.
Since $\Sigma$ is a positive definite matrix, the
condition (\ref{eq:rank}) can be replaced with 
\begin{equation}\label{eq:determinants}
{\rm rank} \,  \Sigma_{A \cup C,  B \cup C}  \leq  \# C.
\end{equation}
without changing the resulting matrices that arise. Consequently, it is natural to use commutative algebra, in particular determinantal varieties, to study the conditional independence structure of a Gaussian random vector.

An important problem in the abstract theory of conditional
independence is to understand when collections of conditional independence
constraints imply other constraints.
For example, for Gaussian random variables, the
two constraints $X_1 \ind X_3$ and $X_1 \ind X_3 \mid  X_2$ imply that
either $X_1 \ind (X_2, X_3)$  or $(X_1, X_2) \ind X_3$  (see Example
\ref{ex:elementary} for a derivation).  While a complete
understanding of such implications is probably impossible
\cite{Sullivant2009}, one hopes that the study of determinantal
varieties and their intersections might shed light upon
conditional independence implications. 

In this direction, one associates to each conditional independence statement $A \ind B \mid  C$  the
\textbf{conditional independence ideal}
$$
J_{A \ind B\mid C}  =  \langle \#C +1 \mbox{ minors of } \Sigma_{A\cup C, B \cup C} \rangle  \subseteq \cc[\Sigma]
$$
and to a list of conditional independence statements, 
$$
\calc =  \{A_{1} \ind B_{1} \mid  C_{1}, A_{2} \ind B_{2} \mid  C_{2} , \cdots  \}
$$
the ideal
$$
J_{\calc}  =  J_{A_{1} \ind B_{1}\mid C_{1}}  + J_{A_{2} \ind B_{2}\mid C_{2}}  + \cdots.
$$
The set of conditional independence statements $\calc$ for Gaussian random variables will imply the conditional independence statement $A \ind B \mid  C$
if 
$$
J_{A\ind B \mid C}  \subseteq  \sqrt{ J_{\calc}}.
$$
Many implications have the form of a disjunction; i.e.~a collection
of conditional independence statements $\calc$ implies that one in a
list of other statements must hold.  Such implications are detected by
computing the primary decomposition of $\sqrt{J_{\calc}}$. 
Given the inherent difficulty of computing primary decompositions
of general ideals, it is natural to look for subfamilies of conditional
independence statements where the ideals $J_{\calc}$ are
radical and the primary decomposition is easy to compute.

In the present paper, we restrict to a family of conditional independence
ideals which define analogs of Fulton's matrix Schubert varieties (see \cite{Fulton}, \cite{Knutson2005}) for symmetric matrices. These \textbf{symmetric matrix Schubert varieties} are indexed by permutations in the symmetric group $S_m$ and are obtained by imposing rank conditions on North-East justified submatrices. These varieties are all reduced, and primary decomposition of their radical defining ideals can be computed using the combinatorics of Bruhat order on $S_m$. Consequently, we obtain a combinatorial algorithm for decomposing conditional independence ideals when we restrict to certain families of ``North-East'' conditional independence constraints.

A second and related problem concerns finding the vanishing ideals of 
\textbf{Gaussian graphical models}.  Specifically, let
$G = ([m], B,D)$ be a mixed graph with $B$ denoting
a set of bidirected edges $i \bi j$ and $D$ denoting a
set of directed edges $i \to j$.  
(This structure can be seen as an undirected graph $([m],B)$
and a directed graph $([m],D)$ sharing a vertex set.)
For the specific results of the present paper, we
will restrict to the case that the set of directed edges forms
a directed acyclic graph, i.e.\ there are no directed cycles.
This means we can reorder the vertices in such a way that
if there is an edge $i \to j \in D$ then $i < j$.

For each  edge $i \to j \in D$ introduce a parameter $\lambda_{ij} \in \mathbb{R}$.
Let $\Lambda$ be the $m \times m$ matrix such that
$$
\Lambda_{ij}  =  \left\{  \begin{array}{cl}
\lambda_{ij}  &  \mbox{ if }  i \to j \in D  \\
0  & \mbox{ otherwise}.
\end{array}  \right.
$$ 
Let $\mathbb{R}^{D}$ denote the set of all such matrices $\Lambda$.
Let $PD_m$ denote the set of $m \times m$ symmetric positive definite
matrices.  Let 
$$
PD(B) :=  \{  \Omega \in PD_m :   \Omega_{ij} = 0 \mbox{ if } i \neq j \mbox{ and }
i \bi j  \notin B  \}.
$$

The Gaussian random vector associated to the mixed graph $G$ with  $\Lambda \in \rr^D$
and $\Omega \in PD(B)$ has $X \sim \mathcal{N}(\mu, \Sigma)$ where
$$
\Sigma  =  (I - \Lambda)^{-T} \Omega (I - \Lambda)^{-1}
$$
and $A^{-T}$ is shorthand for $(A^{-1})^{T}$. 

The Gaussian graphical model associated to $G$ is the set of
such positive definite covariance matrices that can arise:
$$
\mathcal{M}_{G}  =  \{  \Sigma = (I - \Lambda)^{-T} \Omega (I - \Lambda)^{-1}:
\Lambda \in \mathbb{R}^{D},  \Omega \in PD(B) \}. 
$$
One studies the vanishing ideals of these models, denoted 
$$
J_{G}  =  \langle  f \in \mathbb{R}[\Sigma]:  f(\Sigma) = 0 \mbox{ for all } \Sigma \in 
\mathcal{M}_{G}  \rangle.
$$
For general mixed graphs $G$, we currently know no explicit list of ideal generators
of $J_{G}$. In this paper, we find an explicit list of generators in the case where $G$ is a \textbf{generalized Markov chain} (see Definition \ref{def:markovChain}). 
To do this, we identify the varieties $\mathbb{V}(J_{G})$ with symmetric matrix Schubert varieties. This is done by comparing parametrizations for Gaussian graphical models with parametrizations for symmetric matrix Schubert varieties (see Theorem \ref{thm:generalized}).


%

Symmetric matrix Schubert varieties play an important role in both of the above-described statistics problems, and so a substantial portion of this paper is devoted to the study of these varieties and related varieties. We prove results about these varieties that are interesting for their own sake, and not directly related to the conditional independence applications. The paper is structured as follows:

\begin{itemize}
\item In Section \ref{sect:flag}, we consider three spaces of matrices: generic $m\times m$ matrices,
$m\times m$ symmetric matrices, and $m\times m$ upper triangular matrices. 
Each of these spaces of matrices has a stratification naturally induced by the Schubert stratification of an appropriate flag variety. 
The determinantal defining ideal of each closed stratum is obtained by imposing rank conditions on North-East justified submatrices, South-West justified submatrices, and on vertical or horizontal bands of entries, and we write down the prime ideal defining each closed stratum. We also show that each poset of strata, ordered by inclusion, is isomorphic to a type $A$ or $C$ Bruhat interval (see Proposition \ref{prop:stratification}). This material follows naturally from well-known facts about Schubert varieties. We provide a complete explanation so that the reader need not be familiar with the theory of Schubert varieties. 
\item In Section \ref{sect:invertible}, we turn our attention to matrix Schubert varieties. 
These are well-understood in the setting of generic matrices (see \cite{Fulton}, \cite{Knutson2005}): in this case, the defining ideals are prime, the natural generators form a Gr\"obner basis for any diagonal term order, and the intersection of a collection of matrix Schubert varieties is a reduced union of others. We show that the same results hold for matrix Schubert subvarieties of upper triangular\footnote{The upper triangular Gr\"obner basis result can also be obtained by appealing to \cite{WooYongGrobner}.} and symmetric matrices. A key ingredient in our proof is to identify each matrix Schubert variety with a stratum from Section \ref{sect:flag} in order to show that the determinantal defining ideals are prime -- this is one of our main reasons for involving the more general stratification from Section \ref{sect:flag}.
We then give parametrizations of matrix Schubert varieties for symmetric matrices; our parametrizations are analogous to the 
 familiar parametrizations of (generic and upper triangular) matrix Schubert varieties by Chevalley generators (see Propositions \ref{prop:param up}, \ref{prop:param sym}).
\item Our two applications of symmetric matrix Schubert varieties to Gaussian conditional independence appear in Section \ref{sect:applications}. 
\item Though not necessary for our statistics applications, we end the paper with a further analysis of the combinatorics of the three stratifications from Section \ref{sect:flag}. 
Here we describe each defining determinantal ideal, and sums thereof, using certain rank arrays (see e.g.~\cite{Fulton}, \cite{Knutson2005}, \cite{KLS-positroid} for rank arrays in similar settings). 
We also enumerate
the ideals that can appear by permutations with restricted
positions, and give formulas for the number of these components
for generic matrices, symmetric matrices, and upper triangular matrices
(Propositions \ref{prop:Stirling}, \ref{prop:Genocchi}).
A surprise in the upper triangular case is that the number of 
components is enumerated by the median Genocchi numbers.
\end{itemize}

\begin{rmk}
Preliminary versions of this paper made significant use of the theory of Frobenius splitting, and many of our proofs can be rephrased in that language. The interested reader may see \cite[Chapter 2]{BrionKumar}, \cite{KLS}, and \cite{Knutson2009} for relevant related material on Frobenius splitting.
\end{rmk}


\section{Stratifying three spaces of matrices}\label{sect:flag}

Throughout the paper, we let $\mathbb{K}$ be a field. We focus on the following three spaces of matrices: $n\times n$ matrices $\textrm{Mat}_n(\mathbb{K})$, $n\times n$ symmetric matrices $\textrm{Sym}_n(\mathbb{K})$, and $n\times n$ upper triangular matrices, which we denote by $\textrm{Up}_n(\mathbb{K})$. In this expository section, we describe the three stratifications, one for each space of matrices, that are relevant for this paper. The stratifications are obtained by identifying each space of matrices with a Schubert cell or Schubert cell intersected with an opposite Schubert variety, and then stratifying by opposite Schubert varieties.
We begin by setting up some conventions and notation:

\begin{itemize}
\item Let $J_{n}$ be the $n\times n$ matrix with ones along the antidiagonal and zeros elsewhere, let $0_n$ be an $n\times n$ zero matrix, and let $X$, $\Sigma$, and $Y$ be $n\times n$ generic, symmetric, and upper triangular matrices of indeterminates. Define the following $2n\times 2n$ matrices:
\begin{equation}
\widetilde{X} := \begin{bmatrix} J_{n} & X \\ 0_n & J_n \end{bmatrix}, \>\> \widetilde{\Sigma} := \begin{bmatrix} J_n & \Sigma \\ 0_n & J_n \end{bmatrix}, \>\> \widetilde{Y} := \begin{bmatrix} J_n & Y \\ 0_n & J_n \end{bmatrix}. 
\end{equation} 
Let $x_{ij}$, $y_{ij}$, and $\sigma_{ij}$ be the $(i,j)$-entries of the matrices $X$, $Y$, and~$\Sigma$, respectively.
\item Let $v\in S_{2n}$, and let $w_0$ denote the longest word in $S_{2n}$, expressed in one line notation by $2n(2n-1)\cdots 321$. We let $P(v)$ denote the \textbf{permutation matrix} of $w_0v$, so that $P(v)$ is the matrix which has a $1$ in position $(i,j)$ if and only if $w_0v(i) = j$. 
Though this notation is non-standard, we felt it the best choice for our purposes.
\item Let $w_{\square}\in S_{2n}$ denote the ``square word'' permutation which is expressed in one-line notation by $(n+1)\cdots(2n)1\cdots n$, and let $w_{\rm up}\in S_{2n}$ denote the permutation expressed in one-line notation by $12\cdots n(2n)(2n-1)\cdots(n+1)$. Observe that 
\begin{equation}\label{eqn:permutations}
P(w_{\square}) = \begin{bmatrix}J_n&0_n\\0_n&J_n\end{bmatrix},\> \> P(w_{\rm up}) = \begin{bmatrix}0_n&J_n\\I_n&0_n \end{bmatrix}.
\end{equation}
\end{itemize}

Our three stratifications are explicitly described in the following proposition:

\begin{prop}\label{prop:stratification}
\begin{enumerate}
\item There is a prime ideal in $\mathbb{K}[X]$ associated to each $v$ in the Bruhat interval $[1,w_{\square}]\subseteq S_{2n}$ defined by
\[I_{\rm full}(v):=\langle \text{minors of size } (1+\rank P(v)_{[1,i], [j,2n]}) \textrm{ in } \widetilde{X}_{[1,i],[j,2n]} \mid i,j\in [2n]\rangle.
\]
The sum of any two of these ideals is radical and is an intersection of others of the same type, and the poset of these ideals, ordered by inclusion, is isomorphic to the Bruhat interval $[1,w_{\square}]\subseteq S_{2n}$.
\item 
Let $C_{n}$ denote the set of permutations in $S_{2n}$ which satisfy the following condition: if $v = a_1\dots a_{2n}$ in one-line notation, then $a_{i}+a_{2n+1-i} = 2n+1, 1\leq i\leq 2n$. 
There is a prime ideal in $\mathbb{K}[\Sigma]$ associated to each $v\in C_{n}\cap [1,w_{\square}]$ defined by
\[I_{\rm sym}(v):=\langle \text{minors of size } (1+\rank P(v)_{[1,i], [j,2n]}) \textrm{ in } \widetilde{\Sigma}_{[1,i],[j,2n]} \mid i,j\in [2n]\rangle.\]
The sum of any two of these ideals is radical and is an intersection of others of the same type, and the poset of these ideals, ordered by inclusion, is isomorphic to the type $C$ Bruhat interval $[1,w_{\square}]\cap C_n$.
\item There is a prime ideal in $\mathbb{K}[Y]$ associated to each $v$ in the Bruhat interval $[w_{\rm up}, w_{\square}]\subseteq S_{2n}$ defined by:
\[I_{\rm up}(v):=\langle \text{minors of size } (1+\rank P(v)_{[1,i], [j,2n]}) \textrm{ in } \widetilde{Y}_{[1,i],[j,2n]} \mid i,j\in [2n]\rangle.\]
The sum of any two of these ideals is radical and is an intersection of others of the same type, and the poset of these ideals, ordered by inclusion, is isomorphic to the Bruhat interval $[w_{\rm up},w_{\square}]\subseteq S_{2n}$.
\end{enumerate}
\end{prop}

The proof of this proposition is just a matter of identifying each ideal with the scheme-theoretic defining ideal of a \textbf{Kazhdan-Lusztig variety} (i.e. an opposite Schubert variety intersected with a Schubert cell) in an appropriately chosen flag variety. We explain this for items (1) and (3) in Section \ref{sect:typeA}. We justify item (2) in Section \ref{sect:typeB}.

\begin{ex}
When $n=2$, the poset of those $I(v)_{\rm full} \subseteq \mathbb{K}[X]$ from Proposition \ref{prop:stratification} is pictured below:
\begin{center}
{\footnotesize
\begin{tikzpicture}
  \node[draw, align=center] (max) at (0,6) {$I_{\rm full}(1) = \langle 0 \rangle$};
  \node[draw, align=center] (h11) at (-3.5,4) {$I_{\rm full}(s_1) = \langle x_{12}\rangle$};
  \node[draw, align=center] (h12) at (0,4) {$I_{\rm full}(s_2) = \langle \textrm{det} ~X \rangle$};
  \node[draw, align=center] (h13) at (3.5,4) {$I_{\rm full}(s_3) = \langle x_{21} \rangle$};
  \node[draw, align=center] (h21) at (-5,2) {$I_{\rm full}(s_2s_1)$\\ $= \langle x_{11}, x_{12} \rangle$};
  \node[draw, align=center] (h22) at (-2.5,2) {$I_{\rm full}(s_1s_2)$\\ $= \langle x_{12}, x_{22} \rangle$};
  \node[draw, align=center] (h23) at (0,2) {$I_{\rm full}(s_1s_3)$\\$ = \langle x_{12}, x_{21} \rangle$};
  \node[draw, align=center] (h24) at (2.5,2) {$I_{\rm full}(s_2s_3)$\\$ = \langle x_{11}, x_{21} \rangle$};
  \node[draw, align=center] (h25) at (5,2) {$I_{\rm full}(s_3s_2) $\\$= \langle x_{21}, x_{22}\rangle$};
  \node[draw, align=center] (h31) at (-4.5,0) {$I_{\rm full}(s_2s_1s_2)$\\$ = \langle x_{11}, x_{12}, x_{22}\rangle$};
  \node[draw, align=center] (h32) at (-1.5,0) {$I_{\rm full}(s_2s_1s_3)$\\$ = \langle x_{11}, x_{12}, x_{21}\rangle$};
  \node[draw, align=center] (h33) at (1.5,0) {$I_{\rm full}(s_1s_3s_2)$\\$= \langle x_{12}, x_{21}, x_{22}\rangle$};
  \node[draw, align=center] (h34) at (4.5,0) {$I_{\rm full}(s_2s_3s_2)$\\$= \langle x_{11}, x_{21}, x_{22}\rangle$};
  \node[draw, align=center] (min) at (0,-2) {$I_{\rm full}(s_2s_1s_3s_2)$\\$= \langle x_{11}, x_{12}, x_{21}, x_{22}\rangle$};
  \draw (min) -- (h31) -- (h21) -- (h11) -- (max);
  \draw (min) -- (h32);
  \draw (min) -- (h33);
  \draw (min) -- (h34);
  \draw (h31) -- (h22) -- (h12);
  \draw (h32) -- (h21) -- (h12);
  \draw (h32) -- (h23) -- (h11);
  \draw (h32) -- (h24) -- (h12) -- (max);
  \draw (h33) -- (h22) -- (h11);
  \draw (h33) -- (h23) -- (h13) -- (max);
  \draw (h33) -- (h25) -- (h12);
  \draw (h34) -- (h24) -- (h13);
  \draw (h34) -- (h25) -- (h13);
\end{tikzpicture}
}
\end{center}
\noindent Here $s_i$ denotes the simple transposition $(i,i+1)$ and the covering relation is determined by ideal inclusion. It is easy to see that this poset is isomorphic to the Bruhat interval $[1,w_{\square} = 3412]\subseteq S_4$.  Observe also that the length of each permutation determines the codimension of the associated stratum in $\textrm{Mat}_n(\mathbb{K})$. This follows from the fact that each of our strata can be identified with an opposite Schubert variety intersected with a Schubert cell (a.k.a. a Kazhdan-Lusztig variety).

The ideals describing the closed strata of $\textrm{Sym}_n(\mathbb{K})$ and $\textrm{Up}_n(\mathbb{K})$ are described below (on the left and right respectively):

\begin{center}
{\footnotesize
\begin{tikzpicture}
  \node[draw, align=center] (max) at (-3.5,4) {$I_{\rm sym}(1) = \langle 0 \rangle$};
  \node[draw, align=center] (h11) at (-5,2.5) {$I_{\rm sym}(s_1s_3) $\\$= \langle \sigma_{12}\rangle$};
  \node[draw, align=center] (h12) at (-2,2.5) {$I_{\rm sym}(s_2)$\\$ = \langle \textrm{det} ~\Sigma \rangle$};
  \node[draw, align=center] (h32) at (-5,1) {$I_{\rm sym}(s_2s_1s_3)$\\$ = \langle \sigma_{11}, \sigma_{12} \rangle$};
  \node[draw, align=center] (h33) at (-2,1) {$I_{\rm sym}(s_1s_3s_2)$\\$ = \langle \sigma_{12}, \sigma_{22}\rangle$};
  \node[draw, align=center] (min) at (-3.5,-0.5) {$I_{\rm sym}(s_2s_1s_3s_2) $\\$= \langle \sigma_{11}, \sigma_{12}, \sigma_{22}\rangle$};

  \node[draw, align=center] (max2) at (3,4) {$I_{\rm up}(s_3) = \langle 0 \rangle$};
  \node[draw, align=center] (u11) at (1,2.5) {$I_{\rm up}(s_1s_3)$\\$ = \langle y_{12}\rangle$};
  \node[draw, align=center] (u12) at (3,2.5) {$I_{\rm up}(s_2s_3)$\\$ = \langle  y_{11} \rangle$};
  \node[draw, align=center] (u13) at (5, 2.5) {$I_{\rm up}(s_3s_2)$\\$ = \langle y_{22} \rangle$};
  \node[draw, align=center] (u31) at (1,1) {$I_{\rm up}(s_2s_1s_3)$\\$ = \langle y_{11}, y_{12} \rangle$};
  \node[draw, align=center] (u32) at (3,1) {$I_{\rm up}(s_1s_3s_2)$\\$ = \langle y_{12}, y_{22} \rangle$};
  \node[draw, align=center] (u33) at (5,1) {$I_{\rm up}(s_2s_3s_2)$\\$ = \langle y_{11}, y_{22}\rangle$};
  \node[draw, align=center] (min2) at (3,-0.5) {$I_{\rm up}(s_2s_1s_3s_2) $\\$= \langle y_{11}, y_{12}, y_{22}\rangle$};
  
  \draw (min2) -- (u31) -- (u11) -- (max2);
  \draw (u31) -- (u12);
  \draw (min2) -- (u32) -- (u11);
  \draw (u32) -- (u13) -- (max2);
  \draw (min2) -- (u33) -- (u13) -- (max2);
  \draw (u12) -- (max2);
  \draw (u33) -- (u12);
  
  \draw (min) -- (h32);
  \draw (min) -- (h33);
  \draw (h32) -- (h12) -- (max);
  \draw (h32) -- (h11) -- (max);
  \draw (h33) -- (h11);
  \draw (h33) -- (h12);
\end{tikzpicture}
}
\end{center}
The poset on the left is isomorphic to a Bruhat interval of type $C$; in particular, it is isomorphic to the poset of those $v\in S_4$ which commute with the longest word, and which lie in the interval $[1, w_{\square}]\subseteq S_4$. Here the order is induced by ordinary Bruhat order on $S_4$. Again, length of a permutation (with the type $C$ notion of length, so that, in particular, both $s_1s_3$ and $s_2$ have length $1$) determines the codimension of the corresponding stratum in $\textrm{Sym}_n(\mathbb{K})$. The poset on the right is isomorphic to the Bruhat interval $[w_{\rm up} = 1243, w_{\square} = 3412]\subseteq S_4$. In this case, if $v\in S_4$ has length $l(v)$ (for the usual type $A$ notion of length), then the codimension of the corresponding subvariety in $\textrm{Up}_n(\mathbb{K})$ is given by $l(v) - 1$. More generally, the codimension of the stratum of $\textrm{Up}_n(\mathbb{K})$ associated to $v\in S_n$ is given by $l(v) - \#(\textrm{entries below the diagonal})$.
\end{ex}

\begin{rmk}
In the next two subsections, we assume that the field $\mathbb{K}$ is algebraically closed for simplicity. This added assumption is harmless; every property that we need to justify in Proposition \ref{prop:stratification} may be checked over the algebraic closure of the base field instead of over the original base field.
\end{rmk}

\subsection{Type $A$ Schubert varieties and our stratifications of $\textrm{Mat}_n(\mathbb{K})$ and $\textrm{Up}_n(\mathbb{K})$}\label{sect:typeA}

Here we briefly recall some facts about Schubert varieties in type $A$. These well-known results explain items (1) and (3) of Proposition \ref{prop:stratification}. For more information, see one of the many references on the subject (eg. \cite{BrionNotes}, \cite{BrionKumar}, \cite{YoungTableaux}). 

Let $B^+$ (respectively $B^-$) denote the Borel subgroup of upper triangular (resp. lower triangular) matrices in $SL_{2n}(\mathbb{K})$. We will work in the flag variety $B^{-}\backslash SL_{2n}(\mathbb{K})$. \textbf{Schubert cells} are $B^{+}$-orbits and \textbf{Schubert varieties} are the closures of these orbits. Schubert cells and varieties are indexed by permutations in $S_{2n}$: we use $X_v^\circ$ to denote the Schubert cell $B^{-}\backslash B^{-}P(v)B^+$, and will use $X_v$ to denote its closure. 
Similarly, \textbf{opposite Schubert cells} are $B^{-}$-orbits, and \textbf{opposite Schubert varieties} are their closures. These too are indexed by permutations in $S_{2n}$, and 
we let $X^v_\circ$ denote the \textbf{opposite Schubert cell} $B^{-}\backslash B^{-}P(v)B^{-}$, and we let $X^v$ denote its closure. 

The collection of Schubert varieties (resp. opposite Schubert varieties) stratify the flag variety $B^{-}\backslash SL_{2n}(\mathbb{K})$; the intersection of any two Schubert varieties (resp. opposite Schubert varieties) is reduced, and is, scheme-theoretically, a union of other Schubert varieties (resp. opposite Schubert varieties). Containment of strata is determined by Bruhat order on $S_{2n}$. In particular, with respect to our (non-standard) conventions, the containment on opposite Schubert varieties is given by: \[X^v\subseteq X^w\textrm{ if and only if }v>w\textrm{ in Bruhat order}.\]

Of particular interest to us is the Schubert cell $X_{w_{\square}}^\circ$. This Schubert cell is isomorphic to the space of matrices
\[Full := \left\{ \begin{bmatrix}J_{n}&Z \\ 0_{n}& J_{n}\end{bmatrix} ~:~ Z\in Mat_{n}(\mathbb{K})\right\},\]
where, as above, $J_{n}$ denotes the $n\times n$ permutation matrix with $1$s along the antidiagonal and $0$s elsewhere.
The isomorphism from $Full$ to $X_{w_{\square}}^\circ$ is induced by the natural map $SL_{2n}(\mathbb{K})\rightarrow B^{-}\backslash SL_{2n}(\mathbb{K})$. 

The Schubert cell $X_{w_{\square}}^\circ$ has a stratification by opposite Schubert varieties $X^v$, induced by the stratification by opposite Schubert varieties on $B^{-}\backslash SL_{2n}(\mathbb{K})$. In particular, the (closed) strata of $X_{w_{\square}}^\circ$ are all those non-empty intersections $X^v\cap X_{w_{\square}}^\circ$. 
We note that a \textbf{Kazhdan-Lusztig variety} $X^v\cap X_{w_{\square}}^\circ$ is non-empty precisely when $v\leq w_{\square}$ in Bruhat order. 

Each Kazhdan-Lusztig variety is isomorphic to a subvariety of $Full$ obtained by imposing conditions on the ranks of all North-East justified submatrices of
\[ \widetilde{X} = \begin{bmatrix} J_n & X \\ 0_n & J_n \end{bmatrix}.
\]
In particular, by \cite{WooYongSingularities}, the ideal
\[I_{\rm full}(v)=\langle \text{minors of size } (1+\rank P(v)_{[1,i], [j,2n]}) \textrm{ in } \widetilde{X}_{[1,i],[j,2n]} \mid i,j\in [2n]\rangle\] is prime and it scheme-theoretically defines $X_v\cap X^{w_{\square}}_{\circ}$ as a subvariety of $Full$. 
In other words, 
\[X^v\cap X_{w_{\square}}^{\circ} \cong \textrm{Spec } \mathbb{K}[X]/I_{\rm full}(v).\] 
It is useful to note that the ideal $I_{\rm full}(v)$ has a much smaller generating set than the one given above. 
To a permutation matrix $P(v)$ for $v\in S_{2n}$, assign a $2n\times 2n$ grid with a $\times$ in box $(i,j)$ if and only if $P(v)$ has a $1$ in position $(i,j)$. The set of boxes in the grid that have a $\times$ neither directly north nor directly east is the \textbf{diagram} of $v$. Fulton's \textbf{essential set} $\mathcal{E}ss(v)$ is the set of those $(i,j)$ in the diagram such that neither $(i+1,j)$ nor $(i,j-1)$ is in the diagram of $v$. It follows from \cite[\S3]{Fulton} that
\[I_{\rm full}(v) = \langle (1+\rank P(v)_{[1,i], [j,2n]}) \textrm{ in } \widetilde{X}_{[1,i],[j,2n]} \mid (i,j)\in \mathcal{E}ss(v)\rangle.\] 
We refer to this smaller generating set as the set of \textbf{essential minors}.
\begin{proof}[Proof of item (1) of Proposition \ref{prop:stratification}]
This is covered by the background material presented above. In particular, we have that each $I_{\rm full}(v)$, $v\in [1,w_\square]$, is prime. The statement that the sum of any two of these ideals is radical and the intersection of others of the same type is the statement that the intersection of two Schubert varieties is a reduced union of other Schubert varieties. The statement that the poset of these ideals, ordered by inclusion, is isomorphic to the Bruhat interval $[1, w_\square]$ comes from the fact, also mentioned above, that $X^v\subseteq X^w$ if and only if $v>w$ in Bruhat order, together with the fact that $X^v\cap X_{w_{\square}}^\circ$ is non-empty precisely when $v\leq w_\square$ in Bruhat order.
\end{proof}

We next turn our attention to the upper triangular situation. 
Recall that $w_{\rm up}\in S_{2n}$ is the permutation which is expressed in one-line notation by $1\cdots n (2n)\cdots (n+1)$ or by the matrix 
$
P(w_{\rm up}) = \begin{bmatrix} 0_n & J_n \\ I_n &0_n \end{bmatrix}.
$
Then,
\[
\begin{array}{lll}
I_{\rm full}(w_{\rm up}) &= & \langle \text{minors of size } (1+\rank P(w_{\rm up})_{[1,i], [j,2n]}) \textrm{ in } \widetilde{X}_{[1,i],[j,2n]} \mid i,j\in [2n]\rangle\\ 
& = &\langle \textrm{minors of size }(n+1) \textrm{ in } \tilde{X}_{[1,n+i],[i+1,2n]}\mid 1\leq i\leq n-1\rangle \\
& = &\langle x_{ij} \mid j<i \rangle.
\end{array}
\]
Consequently, $I_{\rm full}(w_{\rm up})$ defines the space of matrices
\[
Up:= \left\{ \begin{bmatrix}J_{n} & Z \\ 0_{n} & J_{n}\end{bmatrix} ~:~ Z\in Mat_{n}(\mathbb{K}) \textrm{ is upper triangular }\right\},
\]
as a subvariety of $Full$, and $Up$ is isomorphic to the Kazhdan-Lusztig variety $X^{w_{\rm up}}\cap X_{w_{\square}}^\circ$. 
Because containment of opposite Schubert varieties is given by Bruhat order, the Kazhdan-Lusztig varieties which are isomorphic to subvarieties of $Up$ are those $X^v\cap X_{w_\square}^\circ$ with $v\in [w_{\rm up},w_{\square}]$. Using this, we prove item (3) of Proposition \ref{prop:stratification}.

\begin{proof}[Proof of item (3) of Proposition \ref{prop:stratification}]
We first check that each $I_{\rm up}(v)$ is prime. Since $I_{\rm full}(w_{\rm up}) =  \langle x_{ij}\mid j<i\rangle$, we have that $v\geq w_{\rm up}$ in Bruhat order if and only if $I_{\rm full}(v)$ contains $\langle x_{ij}\mid j<i\rangle$. 
It then follows from the definitions of $I_{\rm full}(v)$ and $I_{\rm up}(v)$ that
\[
\mathbb{K}[X]/I_{\rm full}(v)\cong \mathbb{K}[Y]/I_{\rm up}(v),
\]
for any $v\in [w_{\rm up}, w_{\square}]$. Since $I_{\rm full}(v)$ is prime, so too is $I_{\rm up}(v)$. 

Next, observe that if $v_1, v_2\in [w_{\rm up}, w_\square]$, then 
\begin{equation}\label{eq:radicalUp}
\mathbb{K}[X]/(I_{\rm full}(v_1)+I_{\rm full}(v_2))\cong \mathbb{K}[Y]/(I_{\rm up}(v_1)+I_{\rm up}(v_2))
\end{equation}
and so $I_{\rm up}(v_1)+I_{\rm up}(v_2)$ is radical since $I_{\rm full}(v_1)+I_{\rm full}(v_2)$ is radical. 
Furthermore, since we know that $I_{\rm full}(v_1)+I_{\rm full}(v_2)$ is an intersection of ideals $I_{\rm full}(w_1)\cap \cdots\cap I_{\rm full}(w_r)$ with each $w_i\in [w_{\rm up}, w_\square]$, equation \eqref{eq:radicalUp} yields  
\[
\mathbb{K}[Y]/(I_{\rm up}(v_1)+I_{\rm up}(v_2))\cong \mathbb{K}[X]/(I_{\rm full}(w_1)\cap \cdots\cap I_{\rm full}(w_r)).
\]
To see that $I_{\rm up}(v_1)+I_{\rm up}(v_2)$ can actually be written as $I_{\rm up}(w_1)\cap \cdots\cap I_{\rm up}(w_r)$, identify each upper triangular variable $y_{ij}$ with $x_{ij}\in \mathbb{K}[X]$, $i\leq j$. Observe that each $I_{\rm full}(w_i)$ is generated by the essential minors of $I_{\rm up}(w_i)$ (which involve only those $x_{ij}$ with $i\leq j$) along with the variables $x_{ij}$, $i>j$. So, the intersection $I_{\rm full}(w_1)\cap \cdots\cap I_{\rm full}(w_r)$ has a generating set which consists of generators of the intersection $I_{\rm up}(w_1)\cap \cdots\cap I_{\rm up}(w_r)$ along with the collection of $x_{ij}$, $i>j$. That is, 
\[I_{\rm full}(w_1)\cap \cdots\cap I_{\rm full}(w_r) = (I_{\rm up}(w_1)\cap \cdots\cap I_{\rm up}(w_r))+\langle x_{ij}\mid i>j\rangle,\]
and the desired result follows. 

Finally, observe that the poset of ideals $I_{\rm up}(v)$, $v\in[w_{\rm up}, w_{\square}]$, ordered by inclusion, is isomorphic to the Bruhat interval $[w_{\rm up}, w_{\square}]$. This follows immediately from the fact that the poset of ideals $I_{\rm full}(v)$, $v\in [w_{\rm up}, w_{\square}]$, ordered by inclusion, is isomorphic to the Bruhat interval $[w_{\rm up}, w_{\square}]$.
\end{proof}

\begin{rmk}\label{rmk:essMinorsUpper}
Each ideal $I_{\rm up}(v)$, $v\in [w_{\rm up}, w_\square]$, is generated by its essential minors. This follows since $I_{\rm full}(v)$, $v\in [w_{\rm up}, w_{\square}]$ is generated by its essential minors.
\end{rmk}


\subsection{Type $C$ Schubert varieties and our stratification of $\textrm{Sym}_n(\mathbb{K})$}\label{sect:typeB}

We now turn our attention to Schubert and opposite Schubert varieties in a type $C$ flag variety. We use \cite[Chapter 6]{StdMonTheory} as our reference.

Fix an integer $n\geq 1$ and let $E$ be the $2n\times 2n$ matrix
\[E := \begin{bmatrix}0 &J_n \\ -J_n &0 \end{bmatrix}. \] 
This matrix determines a non-degenerate, skew-symmetric bilinear form on $\mathbb{K}^{2n}$.
The \textbf{symplectic group} $Sp_{2n}(\mathbb{K})$ may then be defined as
\[
Sp_{2n}(\mathbb{K}) = \{A\in SL_{2n}(\mathbb{K}) \mid E(A^{t})^{-1}E^{-1} = A  \},  
\]
or, equivalently, as the fixed point set of the involution
\[
\sigma: SL_{2n}(\mathbb{K})\rightarrow SL_{2n}(\mathbb{K}),\quad
\sigma(A) = E(A^{t})^{-1}E^{-1}.
\]
For convenience, we will sometimes use $A^{-t}$ to denote $(A^t)^{-1}$.

As in \cite{StdMonTheory}, let $H = SL_{2n}(\mathbb{K})$ and let $G = Sp_{2n}(\mathbb{K}) = H^{\sigma}$. 
The Borel subgroups $B_H^+$ and $B_H^{-}$ of upper triangular and lower triangular matrices in $H$ are each stable under $\sigma$, and $B_G^+:=(B_H^+)^{\sigma}$ and $B_G^{-}:=(B_H^{+})^{\sigma}$ are each Borel subgroups of $G$. 
The associated Weyl group $C_n$ can be identified with the set of permutations $v\in S_{2n}$ which commute with the longest word $w_0$. That is, in one-line notation, elements in $C_n$ take the form $a_1\cdots a_na_{n+1}\cdots a_{2n}$ with $a_i+a_{2n+1-i} = 2n+1$ for each $1\leq i\leq 2n$. Bruhat order on $C_n$ is induced by Bruhat order on $S_{2n}$.

We work in the flag variety $B_G^{-}\setminus G$. Schubert cells are $B_G^+$ orbits and Schubert varieties are their closures. Opposite Schubert cells are $B_G^{-}$ orbits and opposite Schubert varieties are their closures. 
Schubert (resp. opposite Schubert) cells and varieties are indexed by Weyl group elements. 
In particular, for $v\in C_n$, we use $X_{G,v}^\circ$ to denote the Schubert cell $B_G^{-}\backslash B_G^{-}P(v)B^+_G$, and use $X_{G,v}$ to denote its closure. Similarly, we let $X^{G,v}_\circ$ denote the opposite Schubert cell $B^{-}_G\backslash B^{-}_GP(v)B^{-}_G$, and we let $X^{G,v}$ denote its closure. 
For $v,w\in C_n$, the inclusion $X^{G,v} \subseteq X^{G,w}$ holds if and only if $v\geq w$ in Bruhat order. 
Finally, the intersection of two opposite Schubert varieties is a reduced union of others (which is a general fact, not special to type $A$ or $C$, following from the Frobenius splitting of the flag variety \cite[Chapter 2]{BrionKumar}). 

\begin{thm}
Let $v\in C_n$. 
\begin{enumerate}
\item \cite[Proposition 6.1.1.1]{StdMonTheory} The Schubert cell $X_v^\circ$ is stable under $\sigma$ and, furthermore, $(X_{H,v}^\circ)^{\sigma} = X_{G,v}^\circ$.
\item \cite[Proposition 6.1.1.2]{StdMonTheory} Under the natural inclusion $B_G^-\backslash G \hookrightarrow B_H^-\backslash H$, there is the following (scheme-theoretic) equality:
\[
X^{G,v} = X^{H,v}\cap B_G^{-}\backslash G.
\]
\end{enumerate} 
\end{thm}

The specific results that we will use now follow. 

\begin{cor}\label{cor:definingEqnsSym}
\begin{enumerate}
\item The space of matrices 
\[Sym := \left\{ \begin{bmatrix}J_n & Z \\ 0 & J_n\end{bmatrix} ~:~ Z\in Mat_{n}(\mathbb{K}) \textrm{ is symmetric }\right\}\]
is isomorphic to the Schubert cell $X_{G, w_\square}^\circ$.
\item Let $v\in C_n$ lie in the Bruhat interval $[1,w_{\square}]$. Then the ideal $I_{\rm sym}(v)\subseteq \mathbb{K}[\Sigma]$ is a non-trivial prime ideal which scheme-theoretically defines the intersection $X^{G,v}\cap X_{G,w_\square}^\circ$. Furthermore, the essential minors generate $I_{\rm sym}(v)$.
\end{enumerate}
\end{cor}

\begin{proof}
We begin with (1), and proceed by showing that $Sym$ is isomorphic to the invariants $(X_{H, w_\square}^\circ)^{\sigma}$. First recall the isomorphism
\[\pi_H: Full \rightarrow X_{H, w_\square}^\circ \] given by identifying a matrix $M$ in $Full$ with its coset $B^{-}M$. 
Observe that, given $x\in X_{H,w_\square}^\circ$, there is a unique $M\in Full$ such that $x = B_H^{-}M$.  It follows that $\sigma(x) = x$ if and only if $\sigma(B_H^{-}M) = B_H^{-}M$. Now, for any $b\in B_H^{-}$, we have that $\sigma(bM) = (E(b^t)^{-1}E^{-1})(E(M^{t})^{-1}E^{-1})$. Since both $B_H^{-}$ and $Full$ are stable under $\sigma$, it follows that $\sigma(B_H^{-}M) = B_H^{-}M$ if and only if $E(M^{t})^{-1}E^{-1} = M$. We now observe that the latter equality is the condition that $M\in Sym$: let $M\in Full$ so that
\[M = \begin{bmatrix}J_n&A\\0&J_n\end{bmatrix}
\]
for some $n\times n$ matrix $A$. Then,
\begin{eqnarray*}
E(M^t)^{-1}E^{-1}&=&\begin{bmatrix}0&J_n\\-J_n&0\end{bmatrix} \begin{bmatrix}J_n&A\\0&J_n \end{bmatrix}^{-t}\begin{bmatrix}0&J_n\\-J_n&0\end{bmatrix}^{-1}\\
&=&\begin{bmatrix}0&J_n\\-J_n&0\end{bmatrix} \begin{bmatrix}J_n&~~0\\-J_nA^tJ_n&~~J_n \end{bmatrix}\begin{bmatrix}0&-J_n\\J_n&0\end{bmatrix}\\
&=&\begin{bmatrix} J_n&A^t\\0&J_n\end{bmatrix},
\end{eqnarray*}
which is equal to $M$ if and only if $A = A^t$. Thus,  $E(M^{t})^{-1}E^{-1} = M$ if and only if $M\in Sym$. 
Consequently, the map $\pi_H: Full \rightarrow X_{H, w_\square}^\circ$ induces an isomorphism from $Sym$ to $(X_{H,w_\square}^\circ)^{\sigma}$.
%
By applying \cite[Proposition 6.1.1.1]{StdMonTheory}, we see that $Sym$ is isomorphic to the Schubert cell $X_{G, w_\square}^\circ$.

We now prove (2).
As in type $A$, $X^{G,v}\cap X_{G,w_\square}^\circ$ is non-empty precisely when $v<w_{\square}$ in Bruhat order on $C_n$, and such non-empty intersections are reduced and irreducible. 
Thus, each variety $X^{G,v}\cap X_{G,w_\square}^\circ$, which is a closed subvariety of the affine space $X_{G,w_\square}^\circ\cong Sym$, has a scheme-theoretic defining ideal which is prime. We now show that this prime ideal is indeed $I_{\rm sym}(v)\subseteq \mathbb{K}[\Sigma]$.

By \cite[Proposition 6.1.1.2]{StdMonTheory}, we have $X^{G,v} = X^{H,v}\cap B_G^{-}\backslash G$ (under the natural inclusion $B_G^-\backslash G \hookrightarrow B_H^-\backslash H$). Then, 
\[X^{G,v}\cap X_{G,w_\square}^\circ = (X^{H,v}\cap B_G^{-}\backslash G) \cap  X_{G,w_\square}^\circ = X^{H,v}  \cap  X_{G,w_\square}^\circ,\]
and we have the following pullback diagram:
\[
\vcenter{\hbox{\begin{tikzpicture}
\node (ul) at (0,2) {$X^{H,v}\cap X_{G,w_\square}^\circ$};
\node (ur) at (4,2) {$X^{H,v}\cap X^\circ_{H,w_\square}$};
\node (ll) at (0,0) {$X^\circ_{G,w_\square}$};
\node (lr) at (4,0) {$X^\circ_{H,w_\square}$};
\path[right hook-latex]
(ul) edge node[above]{} (ur);
\path[right hook-latex]
(ll) edge node[above] {} (lr);
\path[right hook-latex]
(ur) edge node[right] {} (lr);
\path[right hook-latex]
(ul) edge node[right] {} (ll);
\end{tikzpicture}}}
\]
%
Since $\mathbb{K}[X]$ and $\mathbb{K}[\Sigma]$ are the coordinate rings of the Schubert cells $X^\circ_{H,w_\square}$ and $X^\circ_{G,w_\square}$ respectively, and 
$\mathbb{K}[X]/I_{\rm full}(v)$ is coordinate ring of $X^{H,v}\cap X_{H, w_\square}^\circ$, we see that 
\begin{equation}\label{eq:genSet}
\mathbb{K}[X]/I_{\rm full}(v)\otimes_{\mathbb{K}[X]}\mathbb{K}[\Sigma]\cong \mathbb{K}[\Sigma]/I_{\rm sym}(v)
\end{equation}
is the coordinate ring of the affine variety  $X^{H,v}\cap X_{G,w_\square}^\circ$ ($=X^{G,v}\cap X_{G,w_\square}^\circ$). It also follows from \eqref{eq:genSet} that the essential minors of $I_{\rm sym}(v)$ form a generating set of $I_{\rm sym}(v)$, since the ideal $I_{\rm full}(v)$ is generated by its essential minors.
\end{proof}

\begin{proof}[Proof of item (2) of Proposition \ref{prop:stratification}]
The primality of $I_{\rm sym}(v)$ for $v\in [1,w_{\square}]\cap C_n$ is given by Corollary \ref{cor:definingEqnsSym}. The rest of the proof of item (2) of Proposition \ref{prop:stratification} is the same as the proof of the analogous statements in item (1) of Proposition \ref{prop:stratification}.
\end{proof}


\section{Matrix Schubert varieties and their analogs for symmetric and upper triangular matrices}\label{sect:invertible}

In this section, we focus our attention on Fulton's matrix Schubert varieties and their analogs for symmetric and upper triangular matrices.
In each of the generic matrix, symmetric matrix, and upper triangular matrix settings, the associated matrix Schubert varieties are a subset of the strata described in the previous section.
%
%
It turns out that there is a tight relationship between
the matrix Schubert varieties in the three different settings, and that analogs of many standard results on usual matrix Schubert varieties also hold in the cases of symmetric and upper triangular matrices. 
We explain this in this section.

Let $w \in S_{n}$ and let $w_0$ be the longest word in $S_n$. Like in the previous section, we let $P(w)$ be the $n \times n$
permutation matrix of $w_0w$ so that $P(w)$ has a $1$ in location $(i,j)$ iff $w_0w(i) = j$.  
Let $R(w)$ be the $n\times n$ matrix 
such that $R(w)_{ij} =  \sum_{k = 1 }^{i} \sum_{l = j}^{n} P(w)_{kl}$,
the number of $1$'s in $P(w)$ above and to the right.
To the permutation $w$, we associate ideals given by rank conditions
induced by the rank matrix $R(w)$
$$
J_{\rm full}(w) =  \sum_{ij} \langle  R(w)_{ij} + 1 \mbox{ minors of } X_{[1,i],[j,n]}  \rangle \subseteq \mathbb{K}[X].
$$
We can define ideals similarly among symmetric and upper triangular matrices:
$$
J_{\rm sym}(w)  =  \sum_{ij} \langle  R(w)_{ij} + 1 \mbox{ minors of } \Sigma_{[1,i],[j,n]}  \rangle \subseteq \mathbb{K}[\Sigma].
$$
$$
J_{\rm up}(w)  =  \sum_{ij} \langle  R(w)_{ij} + 1 \mbox{ minors of } Y_{[1,i],[j,n]}  \rangle \subseteq \mathbb{K}[Y].
$$

The ideal $J(w)_{\rm full}$ is often called a \textbf{Schubert determinantal ideal}. It is the defining ideal of a matrix Schubert variety. 
In analogy, we call $J_{\rm sym}(w)$ a \textbf{symmetric Schubert determinantal ideal} and call $J_{\rm up}(w)$ an \textbf{upper triangular Schubert determinantal ideal}.

\begin{prop}\label{prop:grobnerInvertible}
For any $w \in S_{n}$, each of the ideals $J_{\rm full}(w)$, $J_{\rm sym}(w)$, and $J_{\rm up}(w)$ are prime. 
Furthermore, with respect to any diagonal monomial order\footnote{By a diagonal monomial order, we mean any monomial order satisfying the property that the leading term of the determinant of a submatrix is the product of the entries along the diagonal of that submatrix.}, the essential minors form a Gr\"obner basis with squarefree initial terms.
\end{prop}

This result is known in the type $A$ setting. 
In the case of generic matrices, 
the Gr\"obner basis result
appeared first in \cite{Knutson2005}. 
The upper triangular case follows in a straight-forward manner from the generic case, or can be seen from results in \cite{WooYongGrobner}. We explain this in the proof below.
A key tool is the following observation about Gr\"obner bases.  Note that
$NF_\mathcal{G}(f)$ denotes the normal form obtained by applying polynomial long
division of the polynomial $f$ by the set of polynomials $\mathcal{G}$, with
respect to the chosen term order.

\begin{lem}\label{grobner:extra}
Let $\mathcal{F} = \{f_1, \ldots, f_r\} \subseteq \kk[x_1, \ldots, x_n, y_1, \ldots, y_n]$ 
be a Gr\"obner basis for an ideal $I$ with respect to a fixed term order $<$
such that $\text{in}_< f_i  \in \kk[x_1, \ldots, x_n]$ for all $i$.
Let  $\mathcal{G}  = \{y_1 - g_1, \ldots, y_n- g_n\}$ be 
a Gr\"obner basis for an ideal $J$
with respect to $<$, such
that $\text{in}_< (y_i - g_i)  = y_i$ for all $i$, and $g_i \in \kk[x_1, \ldots, x_n]$.
Suppose that  $\text{in}_< NF_\mathcal{G}( f_i) = \text{in}_< f_i$ for all $i$.
Then the set 
$$
NF_\mathcal{G}( \mathcal{F})  =  \{  NF_\mathcal{G}(f_1), \ldots,  NF_\mathcal{G}(f_r) \}
$$
is a Gr\"obner basis for the ideal it generates.
\end{lem}

\begin{proof}
First of all, our assumptions about the set of polynomials $\mathcal{F}$ and $\mathcal{G}$
guarantee that the set $\mathcal{F} \cup \mathcal{G}$ forms a Gr\"obner basis
for the ideal they generate.  This is because all S-polynomials formed from pairs of
elements in $\mathcal{F}$ or $\mathcal{G}$ reduce to zero, by the Gr\"obner basis
assumption, and any S-polynomial $S(f_i, y_j-g_j)$ reduces to zero by 
Buchberger's first criterion (see \cite[Ch. 5.5]{BeckerWeispfenning}), since the initial terms of $f_i$ and $y_j-g_j$ are
relatively prime.

Since it has the same leading terms, the set of polynomials  $NF_\mathcal{G}( \mathcal{F}) 
\cup \mathcal{G}$ is a Gr\"obner basis for the same ideal. By our assumptions, the set 
$NF_\mathcal{G}( \mathcal{F})$ is a subset of $\kk[x_1, \ldots, x_n]$.  Let $f \in
\langle NF_\mathcal{G}( \mathcal{F}) \rangle \subseteq \kk[x_1, \ldots, x_n]$.
This polynomial reduces to zero by $NF_\mathcal{G}( \mathcal{F}) 
\cup \mathcal{G}$ since that set is a Gr\"obner basis for the ideal
$\langle NF_\mathcal{G}( \mathcal{F})  \cup \mathcal{G} \rangle$.  However,
since $f \in \kk[x_1, \ldots, x_n]$, it does not have any terms divisible by
any leading term of a polynomial in $\mathcal{G}$.  So $f$ reduces to zero
by applying the division algorithm with $NF_\mathcal{G}( \mathcal{F})$.
This implies that $NF_\mathcal{G}( \mathcal{F})$ is a Gr\"obner basis for 
$\langle NF_\mathcal{G}( \mathcal{F}) \rangle$.
\end{proof}

\begin{proof}[Proof of Proposition \ref{prop:grobnerInvertible}]
We first consider the ideals $J_{\rm full}(w)$. Extend $w\in S_n$ to a permutation $\tilde{w}\in S_{2n}$ by $\tilde{w}(i) = w(i)$, $1\leq i\leq n$, and  $\tilde{w}(i) = i$, $n+1\leq i\leq 2n$. Then, $\tilde{w}\in [1,w_\square]\subseteq S_{2n}$. By definition of $J_{\rm full}(w)$, it is clear that $J_{\rm full}(w)\subseteq I_{\rm full}(\tilde{w})$. For the reverse inclusion, we observe that the essential minors of $I(\tilde{w})_{\rm full}$ agree with the essential minors of $J_{\rm full}(w)$. Thus $J_{\rm full}(w)$ is prime and generated by its essential minors. The Gr\"obner basis statement appears in \cite{Knutson2005}.

Next consider the ideals $J_{\rm up}(w)$. Extend $w\in S_n$ to $\tilde{w}\in S_{2n}$ where this time $\tilde{w}(i) = w(i)$ when $1\leq i\leq n$ and $\tilde{w}(i) = 3n+1-i$ for $n+1\leq i\leq 2n$. 
It is again clear by definition that $J_{\rm up}(w)\subseteq I_{\rm up}(\tilde{w})$. For the reverse inclusion, we use Remark \ref{rmk:essMinorsUpper} along with the observation that the essential minors of $I_{\rm up}(\tilde{w})$ and the essential minors of $J_{\rm up}(w)$ agree. Thus, $J_{\rm up}(w)$ is prime and is generated by its essential minors.


To obtain the Gr\"obner basis result, we apply Lemma \ref{grobner:extra}
where $\mathcal{F}$ is the set of essential minors of $J_{\rm full}(w)$
and $\mathcal{G}$ is the set of variables below the diagonal.  The set
$NF_\mathcal{G}(\mathcal{F})$ consists of the upper triangular minors
which form a Gr\"obner basis by Lemma \ref{grobner:extra}.  Finally substitute
$y_{ij}$ variables for the remaining $x_{ij}$ variables.

Now consider the ideals $J_{\rm sym}(w)$. Extend each $w\in S_n$ to a permutation $\tilde{w}\in C_n\subseteq S_{2n}$ by $\tilde{w}(i) = w(i)$ and $\tilde{w}(2n+1-i) = 2n+1-w(i)$ for $1\leq i\leq n$. Then $\tilde{w}$ defines a prime Kazhdan-Lusztig ideal $I_{\rm sym}(\tilde{w})$. 
In fact $I_{\rm sym}(\tilde{w}) = J_{\rm sym}(w)$. 
Among the generators of $I_{\rm sym}(\tilde{w})$, the minors arising from
$\widetilde\Sigma_{[1,i],[j,2n]}$ are not essential minors when $i,j\leq n$ or $i,j\geq n+1$.
Those with $i\leq n$ and $j\geq n+1$ are the generators of $J_{\rm sym}(w)$,
and those with $i\geq n+1$ and $j\leq n$ replicate these: the minors containing
as many ones as possible from the $J_n$ blocks in $\widetilde\Sigma$ are
the reflections in the main diagonal of $J_{\rm sym}(w)$, up to sign, and the other minors are redundant.
Thus $J_{\rm sym}(w)$ is a prime ideal and is generated by its essential minors.


To obtain the Gr\"obner basis result, we apply Lemma \ref{grobner:extra}
where $\mathcal{F}$ is the set of essential minors of $J_{\rm full}(w)$
and 
$$
\mathcal{G} =  \{ x_{ji} - x_{ij} : 1 \leq  i < j  \leq n\}.
$$  
To make this work, we extend our term order that picks the diagonal leading
terms of the minors whose diagonals are above the main diagonal of $X$ to
a term order on all variables of $X$ by always requiring the $x_{ji} > x_{ij}$ for
$i < j$.  This assumption guarantees that the initial terms remain unchanged on
computing normal forms. Upon computing the polynomials $NF_\mathcal{G}(\mathcal{F})$
and substituting $x_{ij} = \sigma_{ij}$ the resulting set of polynomials consists of the
essential symmetric minors.
\end{proof}

\begin{rmk}
One might wonder if our method of proof extends to show that the essential minors of each ideal $I_{\text{sym}}(v)$ from Section \ref{sect:flag} form a Gr\"obner basis. 
We do not see how to do this: in Proposition~\ref{prop:grobnerInvertible}, the assumption that $w$ is a permutation in $S_n$ is used to 
guarantee that the leading term of each essential minor in $J_{\text{full}}(w)$ consists only of variables lying in the upper triangular part of $X$. This is what allows us to apply Lemma \ref{grobner:extra}.
\end{rmk}

\begin{rmk}\label{rmk:no smS ideals}
The definition of Schubert determinantal ideals is often extended to include ideals determined by North-East ranks of \textbf{partial} permutation matrices. We note that one cannot make the analogous extension in the symmetric and upper triangular settings.  
This is for two reasons: the first is that there exist partial permutation matrices for which there are no upper triangular (resp.\ symmetric) matrices that satisfy the corresponding North-East rank conditions, 
and the second is that even when there are matrices which satisfy the given North-East rank conditions, the corresponding North-East rank ideals often fail to be prime.  We provide examples of this below. 
\end{rmk}

\begin{ex}
Let $u$ be the $2\times 2$ partial permutation matrix
\[
u = \begin{bmatrix}
0&0\\1&0
\end{bmatrix}.
\]
Observe that there are no upper triangular (or symmetric) matrices $M$ which have the property that $M_{[1,i],[j,n]} = w_{[1,i],[j,n]}$, $i,j\in [2]$.

Next consider the partial permutation matrix
\[
v = \begin{bmatrix}0&1\\0&0\end{bmatrix}.
\]
If we extended the definition of upper triangular Schubert determinantal ideal to include ideals defined by partial permutation matrices, then, in this case, we would get the ideal generated by the determinant of $Y$ (i.e. $\langle y_{11}y_{22}\rangle$) which has two components.

Finally, consider the partial permutation matrix
\[w = \begin{bmatrix}0&0&0&1\\0&0&0&0\\1&0&0&0\\0&0&0&0\end{bmatrix}.\]
If we extended the definition of symmetric Schubert determinantal ideal to include ideals defined by partial permutation matrices, then such an ideal for the partial permutation $w$ would be generated by the $2\times 2$ minors of
its submatrices $\Sigma_{[1,2],[1,4]}$ and $\Sigma_{[1,4],[2,4]}$.
This ideal has two prime components.
We provide futher explanation in Example \ref{ex: symmetricDecompose}.
\end{ex}

\begin{prop}\label{prop:bruhat}
Let $\bullet$ denote one of ``full'', ``sym'', or ``up'', and let $v_1, v_2\in S_n$.
The sum $J_\bullet(v_1) + J_\bullet(v_2)$ is reduced, and is the intersection of other ideals of the same type. That is,
\[
J_\bullet(v_1)+J_\bullet(v_2)= J_\bullet(w_1) \cap J_\bullet(w_2) \cap\cdots \cap J_\bullet(w_k)
\]
for some permutations $w_1,\dots, w_k\in S_n$. 
Furthermore, the poset of all $J_\bullet(v)$, ordered by inclusion, is isomorphic to the Bruhat poset of $S_n$.
\end{prop}

\begin{proof}
Both statements are known in the ordinary matrix Schubert variety setting \cite{Knutson2005}. We provide a proof here that works in each case. 

To get the first statement, we extend $v_1$ and $v_2$ in $S_n$ to permutations $\tilde{v}$, and $\tilde{w}$ in $S_{2n}$ in the same way as done in the proof of Proposition \ref{prop:grobnerInvertible}. Then, $J_\bullet(v_i) = I_\bullet(\tilde{v}_i)$. Since each $I_\bullet(\tilde{v}_i)$ is a Kazhdan-Lusztig ideal,
\[
J_\bullet(v_1)+J_\bullet(v_2) = I_\bullet(w_1') \cap I_\bullet(w_2')\cap\cdots\cap I_\bullet(w_k')
\]
where each $I_\bullet(w_i')$ is a Kazhdan-Lusztig ideal. It remains to show that each $w_i'\in S_{2n}$ is actually equal to an extension $\tilde{w_i}\in S_{2n}$ for some $w_i\in S_n$. 
For this, we use a refinement of Proposition~\ref{prop:stratification}(a):
if $\bullet$ denotes ``full'', then the $w_i'$ that appear in the decomposition above
are the least upper bounds for $\tilde v_1$ and~$\tilde v_2$ in Bruhat order.
In particular, all $w_i'$ are less than or equal to $w_\square$.

In the upper-triangular setting, $I_{\rm full}(\tilde v_1)$ and $I_{\rm full}(\tilde v_2)$ both contain 
$\langle x_{ij}\mid i>j \rangle$, whence each of the $I_{\rm full}(w_i')$ do as well.  
As above, the $w_i'$ are also less than or equal to $w_\square$.  
So they must have $w_i'(j) = 3n+1-j$ for each $j\geq n+1$, i.e.\ they must be
of the form $\tilde w_i$.

In the symmetric setting, we claim that all of the $w_i'$ may be taken to lie in the set $C_{n}$
of Proposition~\ref{prop:stratification}(b).  Suppose some $w_i'$ does not.  Under the map
\[
\pi_\Sigma: \mathbb{K}[X] \rightarrow \mathbb{K}[\Sigma], \quad x_{ij} \mapsto \sigma_{ij}
\]
from before, $I_{\rm full}(w_i')$ and $I_{\rm full}(w_0w_i'w_0)$ have the same image,
namely $I_{\rm sym}(w_i')$.  Therefore, in the ``full'' version of the decomposition,
$I_{\rm full}(w_i')$ is equal to $I_{\rm full}(w_i')+I_{\rm full}(w_0w_i'w_0)$,
which in turn is equal to an intersection of ideals of form $I_{\rm full}(w_j'')$ 
for some permutations $w_j''$.
Replace $w_i'$ by the collection of $w_j''$ in the ``sym'' decomposition.  
Since each $w_j''$ is strictly greater than $w_i'$ in Bruhat order and bounded above by $w_\square$,
only a finite number of successive replacements of this form will be possible,
after which point all the permutations involved lie in $C_{n}$.
Finally, being less than or equal to $w_\square$,
these permutations are all of the form $\tilde w_i$.

To obtain the last statement, just note that this is true in the ordinary matrix Schubert variety case \cite{Knutson2005}, and the same inclusion order of ideals holds in the other two cases.
\end{proof}

There are a number of important relationships between the ideals $J_{\rm full}(w)$, $J_{\rm sym}(w),$ $J_{\rm up}(w)$
which will be useful when relating them to Gaussian graphical models.

\begin{prop}\label{prop:degentoup}
For each $w \in S_{n}$, there are weight orders $\tau_{\rm sym}$ and $\tau_{\rm full}$ so that 
$$J_{\rm up}(w)  =  {\rm in}_{\tau_{\rm full}}(J_{\rm full}(w))  \quad \quad \mbox{ and } \quad \quad  
J_{\rm up}(w)  =  {\rm in}_{\tau_{\rm sym}}(J_{\rm sym}(w)).$$
\end{prop}

In Proposition \ref{prop:degentoup}, we mean to take the initial ideal, and then
make an appropriate substitution of variables.

\begin{proof}
For the first statement, we choose the weighting on variables so that 
$\tau_{\rm full}(x_{ij}) = 1 $ if $i \leq j$ and $\tau_{\rm full}(x_{ij}) = 0 $ if $i > j$.
This clearly has the effect in each of the minors that terms above the main
diagonal will have weight the largest weight, so the leading term of a
single such determinant will be the corresponding determinant
of the upper triangular matrices.  This implies that 
$J_{\rm up}(w)  \subseteq  {\rm in}_{\tau_{\rm full}}(J_{\rm full}(w))$.    Since the resulting degeneration
is compatible with the degeneration to the diagonal initial terms
of these minors, and we have a Gr\"obner basis in that case, we deduce that $J_{\rm up}(w)  =  {\rm in}_{\tau_{\rm full}}(J_{\rm full}(w))$.

A similar argument works in the symmetric case, but here we need to take
the weight order $\tau_{\rm sym}(\sigma_{ij})  =  N-|j-i|$, for some very large $N$.  (Actually this would work in the generic case too.)
\end{proof}

Let $V_{\rm full}(w), V_{\rm sym}(w), V_{\rm up}(w)$ denote the vanishing sets of $J_{\rm full}(w)$, $J_{\rm sym}(w)$, and $J_{\rm up}(w)$ respectively. Each of these varieties
have nice parametrizations.  For the ordinary matrix Schubert varieties and in the 
upper triangular case, these are well-known and we repeat them here.
The parametrization for the symmetric case is related to the parametrization
for the upper triangular matrix Schubert variety.  We will need
this to establish the relationship to Gaussian graphical models in
the next section.

Let $s_{1}, \ldots, s_{n-1}$ be adjacent transposition generators of the
symmetric group $S_{n}$.  Let $w\in S_n$ and let $(i_{1}, \ldots, i_{k})$ be a reduced word
for $w_0 w$, in particular $w_0 w = s_{i_{1}} \cdots  s_{i_{k}}$, and $k = \ell(w_0w)$, the length of $w_0w$.
For each $i \in [n-1]$, let $X_{i}(t)$ be the Chevalley generator of the
unipotent group:  $X_{i}(t)  = I  + t e_{i,i+1}$, where $e_{i,i+1}$ is the $n\times n$ matrix with a $1$ in location $(i,i+1)$ and $0$s elsewhere, and where 
$I$ is the $n \times n $ identity matrix.  If $\ell(w_0w) = k$, let
$\phi_{w}:  \cc^{n + k}  \rightarrow  \textrm{Up}_n(\mathbb{C})$ be the map
$$
\phi(a_{1},\ldots, a_{n}, t_{1}, \ldots, t_{k})  =  
{\rm diag}(a_{1}, \ldots, a_{n})  X_{i_{k}}(t_{k})  \cdots  X_{i_{1}}(t_{1}). 
$$
Recall that $\textrm{Up}_{n}(\mathbb{C})$ denotes the set of upper triangular $n \times n$
complex matrices. 

This parametrization is due to Lusztig \cite{lusztig} (see also Fomin and Zelevinsky's result \cite[Theorem 4.4]{fominZelevinsky}). 

\begin{prop}\label{prop:param up}
For a permutation $w$,  the closure of the image of $\phi_{w}$ is the
upper triangular variety $V_{\rm up}(w):= \mathbb{V}(J_{\rm up}(w))$.
\end{prop}

\begin{proof}
This is a straightforward application of \cite[Theorem 4.4]{fominZelevinsky}. After making changes to account for our permutation conventions, their statement specializes to the following: let $w\in S_n$ and let $(i_{1}, \ldots, i_{k})$ be a reduced word
for $w_0 w$. Then,
$$
\phi(a_{1},\ldots, a_{n}, t_{1}, \ldots, t_{k})  =  
{\rm diag}(a_{1}, \ldots, a_{n})  X_{i_{k}}(t_{k})  \cdots  X_{i_{1}}(t_{1}). 
$$
is a biregular isomorphism from $\mathbb{C}^{n+k}_{\neq 0}$ to a Zariski open subset of \[(B_{-}{P(w)}B_{-})\cap B_{+},\] where $B_-$ and $B_+$ denote the sets of lower and upper triangular matrices respectively in $GL_n(\mathbb{C})$. Since the closure of $B_{-}{P(w)}B_{-}\cap B_{+}$ inside of $\textrm{Up}_n(\mathbb{C})$ is $V_{\rm up}(w)$, we get the desired result.
\end{proof}

The description for the symmetric variety arises from taking a transform
of the upper triangular variety.

\begin{prop}\label{prop:param sym}
Let $\psi:   \textrm{Up}_n(\mathbb{C})  \rightarrow  \textrm{Sym}_n(\mathbb{C})$,
$U  \mapsto  U^{T}U$.  Then   $\overline{\psi (V_{\rm up}(w))}  =  V_{\rm sym}(w)$.
\end{prop}

\begin{proof}
First of all, we will show that $\overline{\psi (V_{\rm up}(w))} \subseteq  V_{\rm sym}(w)$.
Indeed, we will show that $J_{\rm sym}(w) \subseteq I( \psi (V_{\rm up}(w)))$.
To see this, note that since $U$ is upper triangular, 
in the matrix $\Sigma = U^{T}U$, we have that
$$
\Sigma_{[1,i],[j,n]}  = U^{T}_{[1,i],[1,i]} U_{[1,i],[j,n]}. 
$$
Since $\rank U_{[1,i],[j,n]}  \leq  R(w)_{ij}$ we deduce that $\rank \Sigma_{[1,i],[j,n]} \leq
R(w)_{ij}$.

Since $w$ is a permutation and $V_{\rm up}(w)$ contains invertible matrices, we know that
$\psi (V_{\rm up}(w))$ contains positive definite matrices.  
The uniqueness of Cholesky decompositions
implies that $\phi$ is generically finite to one on $V_{\rm up}(w)$.
Since $V_{\rm up}(w)$ is a degeneration of $V_{\rm sym}(w)$ (Proposition \ref{prop:degentoup})
we deduce that $V_{\rm up}(w)$, $\psi (V_{\rm up}(w))$, and $V_{\rm sym}(w)$ all have the same dimension.
Since both $\psi (V_{\rm up}(w))$ and $V_{\rm sym}(w)$ are irreducible of the same dimension
and $\psi (V_{\rm up}(w)) \subseteq V_{\rm sym}(w)$, they must be equal.
\end{proof}


\section{Application to Gaussian Conditional Independence Models}\label{sect:applications}

In this section, we explain the applications of matrix Schubert varieties to 
the study of conditional independence structures for
Gaussian random variables.  We use Bruhat order on the symmetric group to study conditional independence implications, as well as characterize the vanishing ideals
of Gaussian graphical models associated to a particular
family of graphs, the generalized Markov chains.

Let $X = (X_{1}, \ldots, X_{n})\sim  \mathcal{N}(\mu, \Sigma)$ be an $n$-dimensional Gaussian random vector, and let $A$, $B$, and $C$ be disjoint subsets of $[n]$. 
Recall from the introduction that one associates to each conditional independence statement $A \ind B \mid  C$  the
\textbf{conditional independence (CI) ideal}
$$
J_{A \ind B\mid C}  =  \langle \#C +1 \mbox{ minors of } \Sigma_{A\cup C, B \cup C} \rangle  \subseteq \cc[\Sigma].
$$

\begin{prop}\label{prop:ci-split}
Let $A \ind B \mid  C$ be a conditional independence statement.  Then the
ideal $J_{A \ind B \mid C}$ is a symmetric Schubert determinantal ideal if and only if one of the following two conditions
are satisfied:
\begin{enumerate}
\item  $A = [1,i]$, $B = [j,n]$ and $C = \emptyset$, for some $i < j$, or
\item  $A = [1,i]$, $B = [j,n]$ and $C = [i+1, j-1]$, for some $i < j-1$.  
\end{enumerate}
\end{prop}

\begin{proof}
Each CI statement $A \ind B \mid C$ gives a rank condition on a submatrix 
of $\Sigma$.  Since CI statements always concern invertible
covariance matrices, we can restrict attention to the invertible
components from Section \ref{sect:invertible}.  
The resulting rank constraints that can arise have the form
$\rank  \Sigma_{[1,i], [j,n]}  \leq r$ for some $i$ and $j$ and
some $r$. 

 There are two ways that such rank conditions
can arise from conditional independence statements.  
If the matrix $\Sigma_{[1,i], [j,n]}$ does not intersect the
diagonal of $\Sigma$, then it must be that $r = 0$, and this corresponds
to the independence statement $[1, i] \ind [j,n]$.
If the matrix $\Sigma_{[1,i], [j,n]}$ does  intersect the
diagonal of $\Sigma$, then $r =  j-i + 1$ (the size of the diagonal overlap)
and we have the conditional independence statement  $[1, j-1] \ind [i+1,n] \mid [j,i]$.
These yield the two cases from the proposition.
\end{proof}

We will sometimes use the term \textbf{Schubert conditional independence ideal} to refer to ideals that are simultaneously symmetric Schubert determinantal ideals and conditional independence ideals.


\begin{rmk}\label{rmk:oneEssBox}
We record here the permutation $w$ associated to each Schubert CI ideal:
\begin{enumerate}
\item if $A$, $B$, and $C$ are as in type (1) of Proposition \ref{prop:ci-split}, then $J_{A\ind B\mid C} = J_{\rm sym}(w)$ where 
\[
w(i) = \left\{
\begin{array}{ll}
\#B+i, & 1\leq i\leq \#A\\
i-\#A, & \#A+1\leq i\leq \#A+\#B\\
i, & \#A+\#B+1\leq i\leq n
\end{array}
\right .
\]
\item if $A$, $B$, and $C$ are as in type (2), then $J_{A\ind B\mid C} = J_{\rm sym}(w)$ where
\[
w(i) = \left\{
\begin{array}{ll}
i, & 1\leq i\leq \#C\\
\#B+i, & \#C+1\leq i\leq \#A+\#C\\
i-\#A, & \#A+\#C+1 \leq i\leq n
\end{array}
\right .
\]
\end{enumerate}
\end{rmk}

In light of Proposition \ref{prop:ci-split}, one can use Bruhat order to deduce conditional implications in certain cases. The following examples illustrate this:


\begin{ex}\label{ex:elementary}
Let $m = 3$ and 
consider the set of conditional independence 
$\mathcal{C} = \{1 \ind 3,  1\ind 3 \mid  2 \}$.
Both of the conditional independence ideal $J_{1 \ind 3}$ and 
$J_{1 \ind 3 \mid 2}$  satisfy the conditions of Proposition
\ref{prop:ci-split} and hence are symmetric Schubert determinantal ideals. 
We have
\begin{eqnarray*}
J_{\mathcal{C}}  & = & J_{1 \ind 3} + J_{1 \ind 3 \mid 2}  \\
                 & =  & \langle \sigma_{13} , \sigma_{12}\sigma_{23} - \sigma_{13}\sigma_{22}  \rangle  \\
                 & = & \langle \sigma_{13} , \sigma_{12}  \rangle \cap
                 \langle \sigma_{13} , \sigma_{23}  \rangle  \\
                 & = &   J_{1 \ind \{2,3\} } \cap J_{ \{1,2\}\ind 3}.
\end{eqnarray*}
For gaussian random variables, we deduce the conditional independence
implication $\{1 \ind 3,  1\ind 3 \mid  2 \}  \implies 1 \ind \{2,3\} \mbox{ or } \{1,2 \}\ind 3$.

In the language of symmetric Schubert determinantal ideals,  
$J_{1 \ind 3} = $$J_{\rm sym}(213)$ 
and $J_{1 \ind 3 \mid 2} = J_{\rm sym}(132)$. 
The decomposition
is equivalent to the decomposition
$J_{\rm sym}(213)+J_{\rm sym}(132) = J_{\rm sym}(312)\cap J_{\rm sym}(231)$.
 \qed
\end{ex}

\begin{ex}
Here we consider a more complex example.  Consider the family of conditional independence statements:
$$
\mathcal{C}  =  \{  1 \ind \{3,4,5\} \mid  2,  \{1,2,3\} \ind 5 \mid  4,  1 \ind \{4,5\}  \}.
$$
Translating into symmetric Schubert determinantal ideals yields:
$$
J_{1 \ind \{3,4,5\} \mid  2} = J_{\rm sym}(15234), 
\quad J_{\{1,2,3\} \ind 5 \mid  4} = J_{\rm sym}(13452), $$
$$
\quad   J_{1 \ind \{4,5\}}  = J_{\rm sym}(31245).
$$
The least upper bounds 
of the set of elements  
$\{15234, 13452, 31245\}$
are the two elements
$\{35142, 51342 \}$
corresponding to the prime decomposition 
$$
J_{\rm sym}(15234)+J_{\rm sym}(13452)+J_{\rm sym}(31245) = J_{\rm sym}(35142)\cap J_{\rm sym}(51342).
$$
Translating the ideals on the righthand side of this equation to
conditional independence yields
$$
J_{\rm sym}(35142)=  J_{\{1, 2 \}\ind \{4, 5\}}  + J_{1 \ind \{3,4,5\} \mid  2} + J_{\{1,2,3\} \ind 5 \mid  4}
$$
$$
J_{\rm sym}(51342) =  J_{1 \ind \{2,3,4,5\}} + J_{\{1,2,3\} \ind 5 \mid  4}. 
$$\qed
\end{ex}

In each of the above examples, we had a sum of Schubert CI ideals equal to an intersection of symmetric Schubert determinantal ideals, where each symmetric Schubert determinantal ideal was itself a sum of Schubert CI ideals. As the next example illustrates, this is not true in general.

\begin{ex}\label{ex:weird}
Consider the family of conditional independence statements
$$
\mathcal{C} = \{ \{1,2\} \ind \{5,6\},  \{1,2\} \ind \{5,6\}  \mid   \{3,4\} \}.
$$
Translating into symmetric Schubert determinantal ideals yields:
$$
J_{\{1,2\}\ind \{5,6\}} = J_{\rm sym}(341256),
\quad J_{ \{1,2\} \ind \{5,6\}  \mid   \{3,4\} } = J_{\rm sym}(125634).
$$
The sum of these two ideals has prime decomposition
$$
J_{\rm sym}(341256)+J_{\rm sym}(125634) = J_{\rm sym}(345612)\cap J_{\rm sym}(561234) \cap J_{\rm sym}(351624).
$$
Now, each of the first two ideals in the above decomposition are CI ideals, namely
$$
J_{\rm sym}(345612) = J_{\{1,2,3,4\}\ind \{5,6\}},
\quad 
J_{\rm sym}(561234) = J_{\{1,2\}\ind \{3,4,5,6\}}.
$$
However, the last ideal in the prime decomposition is \emph{not} a sum of Schubert CI ideals; the permutation matrix $P(351624)$ has an essential box above the main diagonal associated to a rank-$1$ constraint, and this cannot come from a conditional independence ideal (see Proposition \ref{prop:ci-split}). 
\end{ex}

To characterize which symmetric Schubert determinantal ideals are sums of Schubert CI ideals, we recall Gasharov and Reiner's permutations \textbf{defined by inclusions} (see \cite{GasharovReiner} and also \cite[Section 3.1]{Woo}), which we define through the following theorem.

\begin{thm}\label{thm:inclusions}(see \cite[Theorem 4.2]{GasharovReiner} or \cite[Theorem 3.1]{Woo}) 
Let $w\in S_n$. The following are equivalent:
\begin{enumerate}
\item The permutation $w$ is defined by inclusions.
\item If $(i,j)\in \mathcal{E}ss(w)$, then either
\begin{enumerate}
\item there are no $1$s in $P(w)$ weakly North-East of $(i,j)$ (i.e. there is no $k$ with $k\leq i$ and $w(k)\leq j$); or
\item there are no $1$s in $P(w)$ strictly South-West of $(i,j)$ (i.e. there is no $k$ with $k>i$ and $w(k)>j$).
\end{enumerate}
\item The permutation $w$ avoids (in the usual sense of permutation pattern avoidance) $1324$, $31524$, $24153$, and $426153$.
\end{enumerate}
\end{thm}

Note that Theorem \ref{thm:inclusions} has been translated from the versions that appear in \cite{GasharovReiner} and \cite{Woo} to match our unusual permutation conventions.

Suppose $w\in S_n$ is a permutation defined by inclusions. If $(i,j)\in \mathcal{E}ss(w)$ is as in item $2~(a)$ above, then we say it is an essential box of type $1$. If $(i,j)\in \mathcal{E}ss(w)$ is as in $2~(b)$ above \emph{but not also as in $2~(a)$}, we say it is a type $2$ essential box. 

\begin{lem}\label{lem:typeCI}
Let $w\in S_n$. The symmetric Schubert determinantal ideal $J_{\rm sym}(w)$ is equal to a conditional independence ideal $J_{A\ind B\mid C}$ with $A, B, C\subseteq [n]$ if and only if $w$ is defined by inclusions with exactly one element in its essential set $\mathcal{E}ss(w)$. More precisely,  
\begin{enumerate}
\item $\mathcal{E}ss(w)$ has exactly one element and that element is of type $1$ if and only if $A = [1,i]$, $B = [j,n]$, and $C = \emptyset$ for some $i<j$, and
\item $\mathcal{E}ss(w)$ has exactly one element and that element is of type $2$ if and only if $A = [1,i]$, $B = [j,n]$, and $C = [i+1,j-1]$, for some $i<j-1$.
\end{enumerate}
\end{lem}

\begin{proof}
In light of Proposition \ref{prop:ci-split}, it suffices to prove items 1 and 2.  
Remark \ref{rmk:oneEssBox} allows one to easily verify the ``$\impliedby$'' directions of items 1 and 2.
For the other direction of 1, note that if the single element of $\mathcal{E}ss(w)$ is of type 1, then this element must lie strictly above the main diagonal. Suppose it lies in position $(i,j)$ with $i<j$. Then, $J_{\rm sym}(w) = J_{A\ind B\mid C}$ where $A = [1,i]$, $B = [j,n]$, and $C = \emptyset$.
For the ``$\implies$'' direction of 2, note that if the single element of $\mathcal{E}ss(w)$ is of type $2$, then this element $(i,j)$ must lie on or below the main diagonal so that $i\geq j$. Furthermore, $i<n$ and $j>1$ (else $w$ would not be a permutation). One then checks that, 
\[
w(k) = \left\{
\begin{array}{ll}
k, & 1\leq k\leq i-j+1\\
k+n-i, & i-j+2\leq k\leq i\\
k+1-j, & i+1 \leq k\leq n
\end{array}
\right .
\]
Comparing this to the second type of permutation appearing in Remark \ref{rmk:oneEssBox}, we see that in this case $J_{\rm sym}(w) = J_{A\ind B\mid C}$ where $A = [1,j-1]$, $B = [i+1,n]$, $C = [j,i]$.
\end{proof}

\begin{prop}\label{prop:definedbyinclusions}
Let $w\in S_n$, and suppose that $J_{\rm sym}(w)$ is a symmetric Schubert determinantal ideal. Then $J_{\rm sym}(w)$ is a sum of Schubert CI ideals if and only if $w$ is a permutation defined by inclusions.
\end{prop}

\begin{proof}
Suppose $w$ is defined by inclusions. Let $\mathcal{E}ss(w) = \{(i_1,j_1),\dots, (i_r,j_r)\}$.
For each $(i_k, j_k)$, let $w_k$ be the minimal length permutation with $(i_k, j_k)$ in its essential set. This exists and is unique: start with the matrix $P(w)$, and draw one horizontal line and one vertical line to partition $P(w)$ into four blocks so that the South-West corner of the North-East block is $(i_k, j_k)$. Then re-arrange the $1$s within each block so that all $1$s in $P(w)$ appear from North-East to South-West within block rows and within block columns,
without changing the set of columns North of the horizontal line,
or East of the vertical line, containing a~$1$. 
Observe that $(i_k,j_k)$ is the unique element in $\mathcal{E}ss(w_k)$ and that $(i_k,j_k)$ is of type $1$ (respectively type $2$) for $w_k$ whenever $(i_k,j_k)$ is of type $1$ (resp. type $2$) for the original permutation $w$. Furthermore,
\[
J_{\rm sym}(w) = J_{\rm sym}(w_1)+\cdots +J_{\rm sym}(w_r).
\]
By Lemma \ref{lem:typeCI}, each of the above $J_{\rm sym}(w_k)$ is a Schubert CI ideal.

Next suppose that $J_{\rm sym}(w)$ is a sum of Schubert CI ideals. Let $\mathcal{E}ss(w) = \{(i_1,j_1),\dots, (i_r,j_r)\}$, and let $w_k$ be the unique minimal length permutation with the unique essential box $(i_k,j_k)$. Then, 
\[
J_{\rm sym}(w) = J_{\rm sym}(w_1)+\cdots +J_{\rm sym}(w_r).
\]
By assumption, each $J_{\rm sym}(w_k)$ is a Schubert CI ideal, and so $J_{\rm sym}(w_k) = J_{A\ind B\mid C}$ for some disjoint subsets $A,B,C\subseteq [n]$. If $(i_k, j_k)$ is not a type $1$ essential element of $w$, then $C\neq \emptyset$. 
So, by the proof of Lemma \ref{lem:typeCI}, we know that $A = [1,j_k-1]$, $B = [i_k+1,n]$, and $C = [j_k, i_k]$, $i_k\geq j_k$. Thus, the rank of the North-East submatrix of $w_k$ consisting of rows $[1,i_k]$ and columns $[j_k,n]$ is equal to $\#C = i_k-j_k+1$. 
The rank of the corresponding North-East submatrix of $w$ is bounded above by this number. 
Since $w$ is a permutation, the number of $1$s in the North-West submatrix $w_{[1,i_k],[1,j_k-1]}$ is at least $j_k-1$ 
and the number of $1$s in the South-East submatrix $w_{[i_k+1,n],[j_k,n]}$ is at least $n-i_k$. 
This leaves no $1$s left to place strictly South-West of location $(i_k, j_k)$
(and indeed we have placed too many $1$s already unless the North-East ranks
for $w$ and $w_k$ agree). 
Consequently, $(i_k, j_k)$ is a type $2$ essential box of $w$.
\end{proof}

\begin{rmk}
We thank the anonymous referee for pointing out the connection between Schubert varieties defined by inclusions and conditional independence structure. Proposition \ref{prop:definedbyinclusions} appears here because of the referee's comments.
\end{rmk}

\begin{ex}
In Example \ref{ex:weird}, we can see that $J_{\rm sym}(351624)$ is not a sum of symmetric Schubert determinantal CI ideals since $351624$ is not a permutation defined by inclusions. Indeed, $\mathbf{3}5\mathbf{1624}$ does not avoid $31524$ (and this pattern appears in bold in the first permutation).
\end{ex}

We now turn to the problem of finding the vanishing ideals of 
Gaussian graphical models.  Specifically, let
$G = ([m], B,D)$ be a mixed graph with $B$ denoting
a set of bidirected edges $i \bi j$ and $D$ denoting a
set of directed edges $i \to j$. 
We restrict to the case where the set of directed edges forms
a directed acyclic graph, and we reorder the vertices in such a way that
if there is an edge $i \to j \in D$ then $i < j$. 

For convenience, recall from the introduction that for each edge $i \to j \in D$, one introduces a parameter $\lambda_{ij} \in \mathbb{R}$.
Let $\Lambda$ be the $m \times m$ matrix such that
$$
\Lambda_{ij}  =  \left\{  \begin{array}{cl}
\lambda_{ij}  &  \mbox{ if }  i \to j \in D  \\
0  & \mbox{ otherwise}.
\end{array}  \right.
$$ 
Let $\mathbb{R}^{D}$ denote the set of all such matrices $\Lambda$.
Let $PD_m$ denote the set of $m \times m$ symmetric positive definite
matrices.  Let 
$$
PD(B) :=  \{  \Omega \in PD_m :   \Omega_{ij} = 0 \mbox{ if } i \neq j \mbox{ and }
i \bi j  \notin B  \}.
$$
The Gaussian random vector associated to the mixed graph $G$ with  $\Lambda \in \rr^D$
and $\Omega \in PD(B)$ has $X \sim \mathcal{N}(\mu, \Sigma)$ where
$$
\Sigma  =  (I - \Lambda)^{-T} \Omega (I - \Lambda)^{-1}.
$$
and the Gaussian graphical model associated to $G$ is the set of
such positive definite covariance matrices that can arise:
$$
\mathcal{M}_{G}  =  \{  \Sigma = (I - \Lambda)^{-T} \Omega (I - \Lambda)^{-1}:
\Lambda \in \mathbb{R}^{D},  \Omega \in PD(B) \}. 
$$
We are interested in the vanishing ideals of these models, 
$$
J_{G}  =  \langle  f \in \mathbb{R}[\Sigma]:  f(\Sigma) = 0 \mbox{ for all } \Sigma \in 
\mathcal{M}_{G}  \rangle.
$$
While for general mixed graphs $G$, we know no explicit list of generators
of $J_{G}$, there is a natural list of determinantal constraints
associated to the graph, induced by the $t$-separation criterion in $G$, 
as follows.  

\begin{defn}
Let $G = (V,B,D)$ be a mixed graph.  A \textbf{trek} from $i$ to $j$ in $G$
consists of either
\begin{enumerate}
\item a directed path $P_L$ ending in $i$, and a directed path $P_R$ ending in $j$ which
have the same source, or
\item  a directed path $P_L$ ending in $i$, and a directed path $P_R$ ending in $j$
such that the source of $P_L$ and $P_R$ are connected by a bidirected edge.
\end{enumerate}
Let $\mathcal{T}(i,j)$ denote the set of treks in $G$ connecting $i$ and $j$.
\end{defn}

In a trek from $i$ to $j$ we say that the set of vertices in
$P_L$ is the left side and the set of vertices in $P_R$ is the right side
of the trek. Note that in the case that $P_L$ and $P_R$ share a vertex that
vertex is on both sides.

\begin{defn}
Let $G = ([m], B,D)$ be a mixed graph, and let $A_1, A_2, C_1, C_2 \subseteq [m]$
be four not-necessarily disjoint subsets of $[m]$.
The pair $(C_1,C_2)$ is said to \textbf{trek-separate} (or \textbf{t-separate})
$A_1$ and $A_2$ if for all $a_1 \in A$ and $a_2 \in A_2$, every trek in 
$\mathcal{T}(a_1,a_2)$ intersects $C_1$ on the left side or $C_2$ on the right side.
\end{defn}

\begin{thm} \cite{Sullivant2010} \label{thm:tsep}
Let $G = ([m], B,D)$ be a mixed graph and let $A_1, A_2$ be two not necessarily
disjoint subsets of $[m]$.  The matrix $\Sigma_{A_1, A_2}$ has rank at
most $r$ for all $\Sigma \in \mathcal{M}_G$ if and only if there are subsets
$C_1, C_2 \subseteq [m]$ with $\#C_1 + \#C_2 \leq r$  such that $(C_1, C_2)$ 
$t$-separate $A_1$ and $A_2$.
\end{thm}

Theorem \ref{thm:tsep} gives a precise characterization of which minors
of the symmetric matrix $\Sigma$ belong to the vanishing ideal $J_G$.
Note that these determinantal constraints may or may not correspond to
a conditional independence constraint.
It is an interesting open problem to characterize when those determinantal
constraints actually generate $J_G$.  Here we identify a class of
mixed graphs for which $J_G$ is the vanishing ideal of a 
symmetric matrix Schubert variety, which will allow us
to deduce that $J_G$ is generated by determinantal constraints in those
cases.

To do so, we compare the parametrization for graphical models to the parametrization
for the symmetric matrix Schubert varieties.  Note
that for a given permutation $w$ we can write the parametrization
of the symmetric variety $V(w)_{\rm sym}$ via the map:
$$
\phi(a_{1},\ldots, a_{n}, t_{1}, \ldots, t_{k})  =  
X_{i_{1}}(t_{1})^{T} \cdots X_{i_{k}}(t_{k})^{T}{\rm diag}(a_{1}, \ldots, a_{n})  X_{i_{k}}(t_{k})  \cdots  X_{i_{1}}(t_{1}).
$$
On the other hand, we have the parametrization for the Gaussian graphical
model
$$
\Sigma = (I- \Lambda)^{-T} \Omega (I- \Lambda)^{-1}
$$
where $\Lambda \in \mathbb{R}^D$ and $\Omega \in PD(B)$.
To construct examples of Gaussian graphical model that are
symmetric matrix Schubert varieties, we should find
mixed graphs $G$ for which there exists a $w$ such that 
two corresponding parametrizations yield the same matrices.
Because of the structure of the parameterization of 
the symmetric matrix Schubert varieties, we can break this into
two pieces:  
\begin{itemize}
\item  which linear spaces in the space of symmetric matrices
are symmetric matrix Schubert varieties?  and
\item which families of matrices
$\{ (I - \Lambda)^{-1} : \Lambda \in \mathbb{R}^d \}$ can be realized by a parametrization as a product of
Chevalley generators?
\end{itemize}
Taking inverses and noting that the inverse of a Chevalley
generator is a Chevalley generator, we can ask the same thing about the
parametrization of upper triangular
matrices $I - \Lambda$.  This leads us to:

\begin{defn}\label{def:markovChain}
A mixed graph $G = ([n],B,D)$ is called a \textbf{generalized Markov chain}
if 
\begin{enumerate}
\item  $i \to j \in D$ then  $i < j$,
\item  $i \to j \in D$ then for all $i \leq k < l \leq j$ $k \to l \in D$, and
\item  $i \bi j \in B$  and $i < j$ then for all $i \leq k < l \leq j$, $k \bi l \in B$. 
\end{enumerate}
\end{defn}

\begin{thm}\label{thm:generalized}
Let $G$ be a generalized Markov chain.  Then $J_G$ is a symmetric
Schubert determinantal ideal, in particular, $J_G$ is generated by the
determinantal constraints arising from the $t$-separations implied
by the graph $G$.
\end{thm}

\begin{proof}
We proceed to answer the two questions preceding the definition 
of a generalized Markov chain.
First, the linear spaces in the space of symmetric matrices
that are symmetric matrix Schubert varieties must have all rank
conditions determined by setting zero entries in the symmetric matrix.
These patterns of zeros will give rise to a symmetric matrix Schubert variety.
Clearly, we must have that if there is a zero in the $(i,j)$ position, then 
there must also be zeroes everywhere up and to the right from this position.  
This forces that the bidirected part $G$ to satisfy condition (3) in the
definition of a generalized Markov chain.

Second, we need to determine for which directed acyclic graphs $G = ([n], D)$ 
does the parametrization 
$$\xi:  \mathbb{K}^{D}  \rightarrow  \textrm{Up}_n({\mathbb{K}}), ~~ \Lambda \mapsto
I - \Lambda$$
parametrize an upper triangular matrix Schubert variety (technically, because of the
ones on the diagonal, we might call this a unipotent matrix Schubert variety).
To this end, we must ask, when could a pattern of zeros in $I - \Lambda$
arise as the associated set of rank conditions induced by a permutation $w$.  
Clearly, we must have that if there is a zero in the $(i,j)$ position, then 
there must also be zeroes everywhere up and to the right from this position.  Equivalently,
if $(i,j)$ is a nonzero position, then all positions $(k,l)$ where $i\leq k < l \leq j$
must also have nonzero entries.  The corresponding condition on the graph $G$
is that it is a generalized Markov chain. 
\end{proof}

\begin{rmk}
The symmetric and upper triangular matrix Schubert varieties that appear in the above proof are determined by the vanishing of a partition shape of coordinates in the North-East corner. 
Each such matrix Schubert variety is indexed by a $132$-avoiding permutation. In the case of ordinary matrix Schubert varieties, these are pullbacks of Ding's Schubert varieties \cite{ding}.
\end{rmk}

Note that Theorem \ref{thm:generalized} does not give a complete
characterization of the graphs such that $J_G$ is a symmetric Schubert determinantal ideal.  
Indeed, graphs with many edges will just give $J_G = \langle 0 \rangle$
although the graph is not a generalized Markov chain.  For example, 
if $G = ([m], B,D)$ has $B$ consisting of all edges, that $J_G$ is zero regardless
of what $D$ is.

\begin{ex}
Let $G$ be the graph on $[5]$ with only directed edges
$1 \to 2, 2 \to 3, 3 \to 4, 4 \to 5, 1 \to 3, 2 \to 4$.  
Then $G$ is a generalized Markov chain and 
$$
J_{G}  =  J_{\{1,2,3\} \ind 5 \mid  4}  + J_{1 \ind \{4,5\} \mid  \{2,3\}}.
$$ 
\end{ex}

\begin{ex}
Let $G$ be the graph on $[4]$ with directed edges
$1 \to 2, 1\to 3, 2 \to 3, 3 \to 4$, and the bidirected edge
$3 \bi 4$.  This graph is a generalized Markov chain and 
$$
J_G  =  \langle | \Sigma_{12, 34}  |  \rangle.
$$
Note, in particular, that $J_G$ is not a sum of conditional independence ideals.
\end{ex}


\section{On the combinatorics of the three stratifications}\label{sect:combDescription}

We end with some combinatorics of the three stratifications (determined by three Bruhat intervals) from Proposition \ref{prop:stratification}.
\begin{itemize}
\item In subsection~\ref{ssect:explicit}, we make explicit how to decompose sums of ideals from Proposition \ref{prop:stratification} using rank conditions on matrices. 
\item In subsection~\ref{ssect:enumeration}, we provide formulas for the number of strata in our stratifications of $\textrm{Mat}_n(\mathbb{K})$, $\textrm{Up}_n(\mathbb{K})$, and $\textrm{Sym}_n(\mathbb{K})$. In the case of $\textrm{Up}_n(\mathbb{K})$,  the numbers of strata are given by the \textbf{median Genocchi numbers}.
\end{itemize}

\subsection{A lowbrow translation}\label{ssect:explicit}

\newword{Rank arrays} have been used by many authors to study Schubert varieties and related varieties (see eg. \cite{Fulton}, \cite{eriksson}). We make use of them in this section. 
To set our conventions, define a rank array $R$ to be a $2n\times2n$ array of naturals
such that $R_{i,1} = i$ for $1\leq i\leq 2n$ and $R_{2n,j} = 2n+1-j$
for $1\leq j\leq 2n$.
We also make the conventions $R_{0,j} = R_{1,2n+1} = 0$.
A rank array $R$ is \newword{of type~C} if
\[R_{i,j}-R_{2n-i,2n+2-j} = i-j+1\] 
for all $1\leq i,j\leq 2n$.
To each rank array $R$ is associated the \newword{rank ideal}
\[I_{\rm full}(R):=\langle \text{minors of size } (1+R_{ij}) \textrm{ in } \widetilde{X}_{[1,i],[j,2n]} \mid i,j\in [2n]\rangle\]
and, if $R$ is of type~C, also the ideal
\[I_{\rm sym}(R):=\langle \text{minors of size } (1+R_{ij}) \textrm{ in } \widetilde{\Sigma}_{[1,i],[j,2n]} \mid i,j\in [2n]\rangle.\]
We adopt the natural conventions that an ideal of minors larger than the matrix 
containing them is the zero ideal, while an ideal of minors of size 0 or smaller
is the unit ideal.  
The definitions of rank ideals are of course concocted so that a permutation $w\in S_{2n}$
has an associated rank array $R(w)$ given by
\[R(w)_{ij} = \rank P(v)_{[1,i], [j,2n]} = \sum_{k = 1 }^{i} \sum_{\ell = j}^{n} P(w)_{k\ell}\ ,\]
the number of $1$s in $P(w)$ above and to the right of position $i,j$,
which makes $I_{\rm full}(w) = I_{\rm full}(R(w))$ and $I_{\rm sym}(w) = I_{\rm sym}(R(w))$
for all $w$ for which the left sides of these equations are appropriate.

We have not introduced separate rank array technology for the ideals $I_{\rm up}(w)$
as they are a subclass of the ideals $I_{\rm full}(w)$,
so the rank arrays for the latter can also be used to work with the former.

\begin{lem}\label{lem:square}
A permutation $w\in S_{2n}$ lies in $[1,w_{\square}]$ if and only if
$-n\leq (w_i-i)\leq n$ for all $1\leq i\leq 2n$.
\end{lem}

\begin{proof}
The given inequalities on the $w_i$ are equivalent to the 
equalities on the rank array $R(w)$ 
that, for all $1\leq i\leq n$,
$R(w)_{i,n+1-i}=i$ and 
$R(w)_{n+i,2n+1-i}=i$; note that these are the maximal possible values
of these entries of~$R$.
Indeed, it is clear that 
$R(w)_{i,n+1-i}=i$ and $2n+1-w_i\geq n+1-i$
are equivalent in the presence of~$R(w)_{i-1,n+2-i}=i-1$;
the others are symmetric in reflection about the antidiagonal.

Recall that $v\geq w$ in Bruhat order if and only if $X_v \subseteq X_w$,
which itself holds if and only if $R(v)\leq R(w)$ entrywise.
So the content of the lemma is that, to check this containment of~$w$ when $v=w_\square$,
it is enough to check that $R(w)_{i,n+1-i}=i$ and $R(w)_{n+i,2n+1-i}=i$
for each $1\leq i\leq n$.  This is true because 
$R(w_\square)$ is the unique entrywise minimum rank array satisfying
these conditions.  The inequalities 
\[R(w)_{i,j} \leq R(w)_{i+1,j}, R(w)_{i,j-1}, R(w)_{i-1,j}+1, R(w)_{i,j-1}+1\]
hold of the rank array of any permutation, 
and if one uses these iteratively to produce lower bounds for $R(w)$
the bounds one recovers are $R(w_\square)$.
\end{proof}

\begin{rmk}\label{rmk:sjo}
Lemma \ref{lem:square} is a special case of a result of Sj\"ostrand \cite[Theorem 4]{sjo}, which identifies elements in a Bruhat interval $[1,\pi]$, where $\pi$ is a permutation defined by inclusions (see Theorem \ref{thm:inclusions}), with rook placements on an associated skew Ferrers board.
\end{rmk}


For any rank array $R$,
the ideal $I(R)$ is in fact generated by minors of~$X$.  Let
\[I(R)_{i,j} := \langle R_{ij} + 1 \mbox{ minors of } \widetilde X_{[1,i],[j,2n]}\rangle\]
be one of its summands.
It is easy to check that 
\begin{equation}\label{eq:minors of X}
I(R)_{i,j} =
\begin{cases}
\langle R_{ij} + 1 \mbox{ minors of } X_{[1,i],[j-n,n]} \rangle
  & i\leq n, j\geq n+1 \\
\langle R_{ij} - n + j \mbox{ minors of } X_{[n+2-j,i],[1,n]} \rangle
  & i,j\leq n \\
\langle R_{ij} + n + 1 - i\mbox{ minors of } X_{[1,n],[j-n,2n-i]} \rangle
  & i,j\geq n+1 \\
\langle R_{ij} - i + j \mbox{ minors of } X_{[n+2-j,n],[1,2n-i]} \rangle
  & i\geq n+1, j\leq n.
\end{cases}
\end{equation}

The above submatrices of~$X$ are only of positive size
when $n+1<i+j<3n+1$.  For the remaining pairs $(i,j)$ our conventions arrange
that $I(R)_{ij}$ is either zero or the unit ideal.  It is zero for $i+j\leq n+1$
when $R_{i,j}\geq i$, and for $i+j\geq3n+1$ when $R_{i,j}\geq 2n+1-j$, 
and the unit ideal in other cases.
Deeming the unit ideal uninteresting, we lose no other generality in assuming
that $R_{i,j}=i$ for $i+j\leq n+1$ and $R_{i,j}=2n+1-j$ for $i+j\geq 3n+1$.
Let us call these rank arrays \newword{hexagonal}.  The name reflects that the
positions of entries which can vary in a hexagonal rank array form a hexagon.

The next proposition illustrates the connection between the poset of hexagonal rank arrays and the collection of (intersections of) Kazhdan-Lusztig ideals.
As just explained, the ``hexagonal'' assumption of the theorem is not an
important restriction; its use in the proofs will be confined to a few instances.

\begin{prop}\label{prop:intDecRank}
Let $R$ and $S$ be hexagonal rank arrays.  Then 
\begin{enumerate}\renewcommand\theenumi{\alph{enumi}}
\item\label{item:IDRcanonicalise} If
\[R_{i,j} > \min\{R_{i+1,j}, R_{i,j-1}, R_{i-1,j}+1, R_{i,j+1}+1\},\]
and $S$ is identical to~$R$ except that $S_{i,j}$ is equal to the displayed minimum,
then $I_{\rm full}(R)=I_{\rm full}(S)$.
\item\label{item:IDRdecompose} Suppose that there are indices $i$ and~$j$ such that 
\[R_{i,j}=R_{i-1,j}=R_{i,j+1}=R_{i-1,j+1}+1=:r.\]
Let $S$ and $S'$ be arrays identical to $R$ except that
$S_{i-1,j} = r-1$ and $S'_{i,j+1} = r-1$.  Then $I_{\rm full}(R) = I_{\rm full}(S)\cap I_{\rm full}(S')$.
\item\label{item:IDRsum} $I_{\rm full}(R)+I_{\rm full}(S) = I_{\rm full}(\min(R,S))$.
\item\label{item:IDRcontainment} If there are no indices $i$ and~$j$ satisfying the hypotheses
of (\ref{item:IDRcanonicalise}) for $I_{\rm full}(R)$ or $I_{\rm full}(S)$, then
then $I_{\rm full}(R)\subseteq I_{\rm full}(S)$ if and only if $R\geq S$ entrywise.
\item\label{item:IDRprime} If there are no indices $i$ and~$j$ satisfying the hypotheses
of (\ref{item:IDRcanonicalise}) or~(\ref{item:IDRdecompose}) for $I_{\rm full}(R)$, 
then $I_{\rm full}(R)$ is prime.  
\item\label{item:IDRprimarydcp} The primary components of
$I_{\rm full}(R)$ are the minimal ideals $I_{\rm full}(T)$ where $T$ ranges over all arrays
obtained from~$R$ by repeatedly making the replacements specified
in (\ref{item:IDRcanonicalise}) and~(\ref{item:IDRdecompose})
until neither is possible, using either of the two choices in~(\ref{item:IDRdecompose}).
\item\label{item:IDRobtainable} Suppose $I_{\rm full}(R)$ is nonzero.
Then it is a sum of prime ideals of the form $I_{\rm full}(w)$,
as well as an intersection of such prime ideals.
\end{enumerate}

Moreover, suppose $R$ and $S$ are of type~C.  Then the above remain true
when $I_{\rm full}$ is replaced by $I_{\rm sym}$ throughout, the ``identical \ldots\ except for''
clauses in (\ref{item:IDRcanonicalise}) and~(\ref{item:IDRdecompose})
have $S$ and $S'$ differ from~$R$ in two symmetrically-placed entries
to retain the type~C property, and (\ref{item:IDRobtainable}) uses type~$C$ simple reflections.
\end{prop}

\begin{ex}\label{ex: symmetricDecompose}
Consider the hexagonal type~C rank array
\[R = \begin{bmatrix}
1&1&1&1&1&1&1&1\\
2&2&2&2&1&1&1&1\\
3&3&3&2&2&1&1&1\\
4&4&3&2&2&1&1&1\\
5&5&4&3&3&2&2&1\\
6&6&5&4&3&3&2&1\\
7&7&6&5&4&3&2&1\\
8&7&6&5&4&3&2&1
\end{bmatrix}.\]
Pursuant to Remark~\ref{rmk:no smS ideals},
the ideal $I_{\rm sym}(R)$ is the analogue of the North-East rank ideal $J_{\rm sym}(w)$ where $w$ is the partial permutation with matrix
\[\begin{bmatrix}0&0&0&1\\0&0&0&0\\1&0&0&0\\0&0&0&0\end{bmatrix}.\]
It is generated by the $3\times 3$ minors of $\Sigma$ and the $2\times 2$ minors of
its submatrices $\Sigma_{[1,2],[1,4]}$ and $\Sigma_{[1,4],[2,4]}$.

By part~(\ref{item:IDRobtainable}), the ideal $I_{\rm sym}(R)$ is a sum of prime ideals of the form $I_{\rm sym}(w)$,
which can be recovered by the procedure in part~(\ref{item:IDRprimarydcp}).
By part (\ref{item:IDRprime}), $I_{\rm sym}(R)$ is not itself prime,
since the replacement in part~(\ref{item:IDRdecompose}) is possible,
for $(i,j) = (3,4)$ and $(6,5)$.  Concretely, take $(i,j) = (3,4)$.
Then the two resulting arrays, called $S$ and $S'$ in (\ref{item:IDRdecompose}), are respectively
\[
R_1 = \begin{bmatrix}
1&1&1&1&1&1&1&1\\
2&2&2&\underline1&1&1&1&1\\
3&3&3&2&2&1&1&1\\
4&4&3&2&2&1&1&1\\
5&5&4&3&3&2&2&1\\
6&6&5&4&3&\underline2&2&1\\
7&7&6&5&4&3&2&1\\
8&7&6&5&4&3&2&1
\end{bmatrix}
\mbox{ and }
R_2 = \begin{bmatrix}
1&1&1&1&1&1&1&1\\
2&2&2&2&1&1&1&1\\
3&3&3&2&\underline1&1&1&1\\
4&4&3&2&2&1&1&1\\
5&5&4&3&\underline2&2&2&1\\
6&6&5&4&3&3&2&1\\
7&7&6&5&4&3&2&1\\
8&7&6&5&4&3&2&1
\end{bmatrix},
\]
in both of which two (underlined) entries have been replaced following the prescription at the end of the proposition,
and we have $I_{\rm sym}(R) = I_{\rm sym}(R_1)\cap I_{\rm sym}(R_2)$.
Now $I_{\rm sym}(R_1)$ is prime by part~(\ref{item:IDRprime}), while $I_{\rm sym}(R_2)$ is subject to 
further replacements as in part~(\ref{item:IDRdecompose}), namely at $(i,j) = (4,4)$ or $(6,4)$.  
Taking the former, we produce
\[
R_3 = \begin{bmatrix}
1&1&1&1&1&1&1&1\\
2&2&2&2&1&1&1&1\\
3&3&3&\underline1&1&1&1&1\\
4&4&3&2&2&1&1&1\\
5&5&4&3&2&\underline1&2&1\\
6&6&5&4&3&3&2&1\\
7&7&6&5&4&3&2&1\\
8&7&6&5&4&3&2&1
\end{bmatrix}
\mbox{ and }
R_4 = \begin{bmatrix}
1&1&1&1&1&1&1&1\\
2&2&2&2&1&1&1&1\\
3&3&3&2&1&1&1&1\\
4&4&3&2&\underline1&1&1&1\\
5&5&4&3&2&2&2&1\\
6&6&5&4&3&3&2&1\\
7&7&6&5&4&3&2&1\\
8&7&6&5&4&3&2&1
\end{bmatrix},
\]
with $I_{\rm sym}(R_2) = I_{\rm sym}(R_3)\cap I_{\rm sym}(R_4)$.  
Now $I_{\rm sym}(R_4)$ is prime, 
while $I_{\rm sym}(R_3)$ is subject to the replacements of the sort in part~(\ref{item:IDRcanonicalise}),
in fact twice in succession, with $(i,j) = (2,4)$ and then $(3,3)$.  These yield
\[
R_5 = \begin{bmatrix}
1&1&1&1&1&1&1&1\\
2&2&2&\underline1&1&1&1&1\\
3&3&\underline2&1&1&1&1&1\\
4&4&3&2&2&1&1&1\\
5&5&4&3&2&1&\underline1&1\\
6&6&5&4&3&\underline2&2&1\\
7&7&6&5&4&3&2&1\\
8&7&6&5&4&3&2&1
\end{bmatrix}
\]
with $I_{\rm sym}(R_3) = I_{\rm sym}(R_5)$.  By part~(d), $I_{\rm sym}(R_1)\subsetneq I_{\rm sym}(R_5)$,
so the latter can be excluded from the primary decomposition, but there are no other comparabilities between the primes produced.  
We conclude that the primary decomposition of $I_{\rm sym}(R)$ is
\[I_{\rm sym}(R) = I_{\rm sym}(R_1)\cap I_{\rm sym}(R_4) = I_{\rm sym}(16472538)\cap I_{\rm sym}(15672348).\]
\end{ex}

\begin{proof}[Proof of Proposition~\ref{prop:intDecRank}]
We prove the parts out of order.  In particular the proof of part~(\ref{item:IDRdecompose}) 
will come late; don't be misled by the fact that we refer to the operation
on arrays it invokes before proving what effect this replacement has on ideals.
  
For both $I_{\rm full}$ and $I_{\rm sym}$,
part~(\ref{item:IDRcanonicalise}) follows from containments among ideals of minors, since
the $r$-minors of $\tilde X_{[1,i],[j,2n]}$ are included among
the $r$-minors of a larger matrix, and are generated by the
$(r-1)$-minors of a matrix one row or column smaller, by expansion
along that row or column.

Part~(\ref{item:IDRsum}) is also easy,
since the minors generating $I_{\rm full}(R)+I_{\rm full}(S)$ appear among those generating
$I_{\rm full}(\min(R,S))$ and the other generators of the former are larger minors,
which are redundant.  For $I_{\rm sym}$ the argument is the same.

We will show part~(\ref{item:IDRprime}) by
comparing to the description of the
Kazhdan-Lusztig varieties in Proposition~\ref{prop:stratification}.
Define a $2n\times2n$ array $w$ by
\[w_{i,j}=R_{i,j}-R_{i,j+1}-R_{i-1,j}+R_{i-1,j+1}.\]
If no replacement as in part~(\ref{item:IDRcanonicalise}) is possible in~$R$,
then $R_{i,j}-R_{i,j+1}\in\{0,1\}$, so that $w_{i,j}$, which is the
difference of two such differences, equals $-1$ or 0 or~1.
If further no replacement as in part~(\ref{item:IDRdecompose}) is possible,
the only state of affairs allowing $w_{i,j}=-1$ is ruled out, so $w_{i,j}\in\{0,1\}$.
Also
\[\sum_{i=1}^{2n}w_{i,j} = R_{i,j}-R_{i,j+1}-R_{i-1,j}+R_{i-1,j+1} = 1\]
so there is a unique 1 in every column of~$w$; by a symmetric argument
there is a unique 1 in every row of~$w$.  So $w$ is a permutation matrix.
Lastly, the hexagonality condition implies that no 1 in~$w$ has
$|w_i-i|>n$, so that $w\in[1,w_\square]$ by Lemma~\ref{lem:square}.
So $I(R) = I(w)$ is one of the prime ideals of Proposition~\ref{prop:stratification}.
As for $I_{\rm sym}$, if $R$ is of type~$C$ then from the
definition we get $w_{i,j} = w_{2n+1-i,2n+1-j}$, which 
implies that $w$ commutes with $w_0$ so that 
$I_{\rm sym}(R) = I_{\rm sym}(w)$ is again a prime ideal
appearing in Proposition~\ref{prop:stratification}.

To prove part~(\ref{item:IDRobtainable}), we will show that
every nonzero ideal $I_{\rm full}(R)$ or $I_{\rm sym}(R)$ is a sum of prime ideals
of the form $I_{\rm full}(w)$ or $I_{\rm sym}(w)$ accordingly.  This implies
that $I_{\rm full}(R)$ resp.\ $I_{\rm sym}(R)$ 
is the ideal of an intersection of Kazhdan-Lusztig varieties, by part~(\ref{item:IDRsum}).
Such intersections are unions of Kazhdan-Lusztig varieties as well.  
To identify them precisely,
the cells $X_v^\circ\cap X^{w_\square}_\circ$ that appear in an intersection
$\bigcap_i (X_{w_i}\cap X^{w_\square}_\circ)$ are just those where $v$
is greater than or equal to each $w_i$ in Bruhat order,
and less than or equal to $w_\square$.
Thus $I_{\rm full}(R)$ is the intersection of any collection of $I_{\rm full}(w)=I_{\rm full}(R(w))$
where $w$ ranges over a set of permutations including all the minimal such~$v$
and perhaps some greater ones.
The analogue is also true for cells 
$X_{G,v}^\circ\cap X^{G,w_\square}_\circ$
in the notation of Section~\ref{sect:typeB}, giving the type $C$ statement.

We begin with $I_{\rm full}(R)$.  We may assume to start with that $R$ allows
no replacements as in part~(\ref{item:IDRcanonicalise}), since making them leaves $R$ unchanged.
Our claim is that by making replacements
``dual'' to those in parts (\ref{item:IDRcanonicalise}) and~(\ref{item:IDRdecompose}),
we eventually resolve $R$ into a collection of rank arrays $R(w)$
where these substitutions are no longer possible.  The precise replacements
in question are these: if
\begin{equation}\label{eq:IDRdual1}
R_{i,j}<\max\{R_{i+1,j}-1,R_{i,j-1}-1,R_{i-1,j},R_{i,j+1}\}
\end{equation}
then change the array by incrementing the $(i,j)$ entry;
and if
\begin{equation}\label{eq:IDRdual2}
R_{i,j}=R_{i-1,j}=R_{i,j+1}=R_{i-1,j+1}+1=r
\end{equation}
then make two altered copies of the array, one in which the $(i,j)$ entry is incremented
and another in which the $(i-1,j+1)$ entry is.

Observe that these replacements are possible in exactly the same conditions as 
are those of parts (\ref{item:IDRcanonicalise}) and~(\ref{item:IDRdecompose}),
so that the process will indeed terminate at rank arrays of permutations by part~(\ref{item:IDRprime}).
Let the resulting arrays be $S_1,\ldots,S_m$.  Each $S_k$ is entrywise greater than or equal to~$R$
by construction. 

Now, by assumption the first replacement made to~$R$ is one of type~\eqref{eq:IDRdual2}.
Let $S$ and $S'$ be the rank arrays obtained after making all possible replacements 
of type~\eqref{eq:IDRdual1} on the results thereof, so that in $S$ the $(i,j)$ entry
has been incremented and in $S'$ the $(i-1,j+1)$ entry has.  Consider~$S$.
It is not difficult to show inductively,
using the fact that replacements \eqref{eq:IDRdual1} did not apply to~$R$,
that in the course of producing $S$,
after an entry at position $(i',j')$ is incremented, the only entries newly in the conditions \eqref{eq:IDRdual1}
and therefore subject to be incremented are those at positions $(i'-1,j')$ and $(i',j'+1)$.
Therefore, all entries in which $S$ differs from~$R$ are weakly South-West of~$(i,j)$.

Similarly, all entries in which $S'$ differs from~$R$ are weakly North-East of $(i-1,j+1)$.
So these sets of entries are disjoint, and in each position
at least one of $S$ and~$S'$ agrees with~$R$.  
By induction on the number of replacements, for every position $(i,j)$ in~$R$ there exists a~$k$
so that $(S_k)_{i,j}=R_{i,j}$.  It follows that $R$ is the entrywise minimum of all the $S_k$.
Therefore, by part~(\ref{item:IDRsum}), the ideal $I_{\rm full}(R)$ is a sum of Kazhdan-Lusztig prime ideals,
which is what we needed above.

In adjusting this argument for $I_{\rm sym}$, we must make the replacements
\eqref{eq:IDRdual1} and~\eqref{eq:IDRdual2} in symmetrical pairs to retain the type~$C$ property.
Then the only remaining thing to check is that that $S$ and~$S'$ still differ in
disjoint sets of entries.  Let us take the case that $(i,j)$ sits above the main diagonal.
The entries of~$S'$ which differ from~$R$ are still all restricted to the positions 
weakly North-East of $(i-1,j+1)$, or the mirror image, weakly South-West of $(2n+1-i,2n+1-j)$.
So we need to establish that the entries of $S$ which differ from~$R$ don't impinge
on the positions weakly South-West of $(2n+1-i,2n+1-j)$.  But these entries are all
$r$s in~$R$ which get changed to $(r+1)$s in~$S$, and form a Ferrers shape with
its corner in the South-West,
so neither of the positions $(2n+1-i,2n+1-j)$ or $(2n+1-i,2n+1-j)$ may be among them because 
$R_{2n-i,2n+2-j} = R_{2n+1-i,2n+2-j}-1 = R_{2n-i,2n+1-j}-1$, 
and thus neither may $(2n+1-i,2n+1-j)$ or positions further South-West.

For part~(\ref{item:IDRcontainment}), it is clear from the generators
that $I_{\rm full}(R)\subseteq I_{\rm full}(S)$ when $R\geq S$.  Inversely, if $R\not\geq S$,
it is enough to produce primes $I_{\rm full}(v)$ contained in~$I_{\rm full}(R)$ and $I_{\rm full}(w)$ containing $I_{\rm full}(S)$
for which $I_{\rm full}(v)\not\subseteq I_{\rm full}(w)$.  Choose $(i,j)$ so that $R_{i,j}<S_{i,j}$.
As argued above, there is a way of repeatedly applying replacements 
\eqref{eq:IDRdual1} and~\eqref{eq:IDRdual2} to~$R$ 
to yield the rank array $R(v)$ of a prime ideal without changing the value at position $(i,j)$.
In the ``dual'' fashion $R(w)$ can be produced from $S$ by applying the replacements
(\ref{item:IDRcanonicalise}) and~(\ref{item:IDRdecompose}) without changing the value at position $(i,j)$.
But then $R(v)_{i,j}<R(w)_{i,j}$ 
implies $v\not\geq w$ in Bruhat order, so $I_{\rm full}(v)\not\subseteq I_{\rm full}(w)$ as desired.
The argument for $I_{\rm sym}$ is exactly analogous.

We move on to part~(\ref{item:IDRdecompose}).  Using part~(\ref{item:IDRobtainable})
it is enough to show that a prime $I_{\rm full}(w)$ contains $I_{\rm full}(R)$ 
if and only if it contains either $I_{\rm full}(S)$ or $I_{\rm full}(S')$, or the same for $I_{\rm sym}$.
Using part~(\ref{item:IDRcontainment}), this translates to showing that
a rank array $R(w)$ is less than $R$ if and only if it's less than one of $S$ or~$S'$.
The ``only if'' direction follows because $S$ and $S'$ are less than $R$.
The ``if'' direction holds because no $R(w)$ can agree in all four entries
$(i,j)$, $(i,j+1)$, $(i-1,j)$ and $(i-1,j+1)$ with $R$ by part~(\ref{item:IDRprime}),
since a replacement of type~(\ref{item:IDRdecompose}) is possible otherwise,
so one of these entries must be lesser; in fact one of the $(i,j+1)$ or $(i-1,j)$
entries must be lesser, since a replacement of type~(\ref{item:IDRcanonicalise})
is possible if only one or both of the $(i,j)$ and $(i-1,j+1)$ entries are.

Finally, part~(\ref{item:IDRprimarydcp}) is an immediate consequence of
parts (\ref{item:IDRcanonicalise}), (\ref{item:IDRdecompose}), and~(\ref{item:IDRprime}).
\end{proof}

\subsection{Enumeration}\label{ssect:enumeration}
Let $x(n)$, $y(n)$, and $\sigma(n)$ denote the number of irreducible strata in the stratifications of $Mat_n(\mathbb{K})$, $Up_n(\mathbb{K})$, and $Sym_n(\mathbb{K})$ from Proposition \ref{prop:stratification}. Equivalently, $x(n)$ is the number of permutations in the Bruhat interval $[1,w_\square]\subseteq S_{2n}$, $y(n)$ is the number of permutations in the Bruhat interval $[w_{\rm up}, w_\square]\subseteq S_{2n}$, and $\sigma(n)$ is the number of permutations in $[1, w_\square]\cap C_n$. 
A noteworthy feature of these sets of permutations 
is that the Bruhat order conditions on them may all be reinterpreted as
\textbf{rook placement} problems on variously shaped boards.
See \cite[Ch.~2]{Stanley2012} for a basic overview of enumeration of rook placements.

\begin{rmk}
Interpreting certain Bruhat intervals in terms of rook placements is not a new idea. For example, Sj\"ostrand did this for certain intervals $[1,\pi]$ in \cite{sjo} (as we discussed in Remark \ref{rmk:sjo}). Furthermore, in light of results in \cite{sjo}, it is not surprising that rook placements play an important role in the proof of our next result. 
\end{rmk}

\begin{prop}\label{prop:Stirling}
The number $x(n)$ 
is given by
$$
x(n)  =  \sum_{k = 1}^{n+1}  ( (k-1)!\, S(n+1, k))^{2}
$$
where $S(n,k)$ denotes the Stirling number of the second kind.
The number $\sigma(n)$ 
is given by
$$
\sigma(n)  =  \sum_{k = 1}^{n+1}   (k-1)!\, S(n+1, k).
$$
\end{prop}

\begin{proof}
According to Lemma \ref{lem:square}, in order to find the number $x(n)$ 
we need to count the permutations $w \in S_{2n}$
such that 
\begin{equation}\label{eq:permineq}
| w(i) - i |  \leq n
\end{equation}
  For the symmetric case,
we need permutations in $S_{2n}$ satisfying the same condition
plus the extra symmetry condition $w(i) + w(2n+1 -i)  = 2n+1$.

To explain one solution to this counting problem, 
we consider the problem of constructing the 
permutation matrices that satisfy these conditions.
We represent these permutation matrices as block matrices 
$$
P(w) =
\begin{pmatrix}
A & B \\
C & D
\end{pmatrix}.
$$
The condition (\ref{eq:permineq}) on a permutation $w$,
recalling our conventions for permutation matrices from Section~\ref{sect:flag}, 
is equivalent to requiring that 
the nonzero entries of matrix $A$ fall on or below the antidiagonal,
and those of matrix $D$ fall on or above the antidiagonal.
Note that taking an allowable permutation $w$ and looking
at the induced matrix $A$ gives a rook placement of some  number $k$ of
 non-attacking rooks on a Ferrers board of staircase shape, i.e.\
with $1$ box in the first row, $2$ boxes in the second, $\ldots$, and $n$ boxes in the
$n$\/th row.  According to Corollary 2.4.2 of \cite{Stanley2012}, 
the number of such rook placements with $k$ rooks is $S(n+1, n+1 -k)$.
Note that if $A$ has exactly $k$ ones in it, $D$ must also have exactly
$k$ ones in it, and by symmetry also gives a nonattacking rook placement
in a tableau of the same shape.  Finally, once we have specified $A$ and
$D$ each with $k$ ones, we get unique $(n-k) \times (n-k)$ submatrices
of $B$ and $C$ where we can put a permutation matrix, to complete
the permutation matrix $P_{w}$.  Summing over $k$ yields the total
number of permutations
$$
\sum_{k = 0}^{n}  ((n-k)!\,S(n+1, n+1 -k))^{2},
$$
which is the desired formula after reindexing.

In the symmetric case choosing one matrix $A$ completely determines
the matrix $D$, and the matrix $B$ determines the matrix $C$.  This removes
the square from the formula.
\end{proof}

\begin{prop}\label{prop:Genocchi}
The number $y(n)$ 
is the median Genocchi number $H_{2n+1}$.
\end{prop}

Let $P(n)$ be the poset on $\{1, \ldots, n, \bar 1, \ldots, \bar n\}$ whose covering
relations are those of the forms $\bar i \succdot i$, $i+1 \succdot i$, and $\overline{i+1} \succdot \bar i$.
That is, $P(n)$ is the product of chains of lengths $n$ and~2.  
Let $\widetilde P(n)$ be the poset on the larger set $\{1, \ldots, n+1, \bar 0, \ldots, \bar n\}$
with the order defined in the analogous way.

\begin{lem}\label{lem:Genocchi}
The prime components counted by $y(n)$ are in bijection with the set
\[G(n) = \left\{\begin{minipage}{24em}
ordered lists of disjoint chains in $P(n)$,\\ each chain containing
both barred and unbarred elements\end{minipage}\right\}.\]
\end{lem}

\begin{proof}
According to Proposition~\ref{prop:stratification} (3), $y(n)$
is the number of permutations in the Bruhat interval $[w_{\rm up},w_\square]$.
As before, the matrices of these permutations can be interpreted
as the solutions to a rook placement problem on a certain board.
The disallowed cells are those $(i,j)$
where $|i-(2n+1-j)|>n$, enforcing the upper bound of the interval, and
those where $i+(2n+1-j)\leq n$, enforcing the lower bound.
Figure~\ref{fig:Genocchi rook} depicts an example.

\begin{figure}
\includegraphics{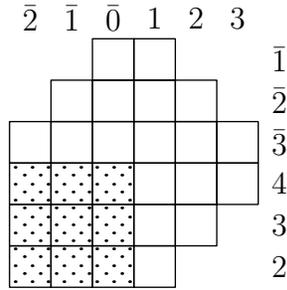}
\caption{A board for the rook placement problem for $y(3)$.}\label{fig:Genocchi rook}
\end{figure}

%

Given a rook placement on such a board, we bijectively construct an
element of~$G(n)$.  Label the rows and columns of the board with elements of $\widetilde P(n)$:
the rows get labels
$n, \ldots, 1,\linebreak[0] \bar 0, \ldots, \overline{n-1}$ right to left,
and the columns $\bar 1, \ldots, \bar n,\linebreak[0] n+1, \ldots, 2$ top to bottom,
again as depicted in Figure~\ref{fig:Genocchi rook}.  
Now construct the finest partition of $\widetilde P(n)$ such that,
for every rook not in the lower-left (speckled) quadrant
of the board, its row and column indices lie in the same part.  
Because the column label exceeds the row label for each of these rooks,
this gives a partition of $\widetilde P(n)$ into chains.  
Discarding the chains containing $n+1$ and $\bar 0$ leaves a set $g$
of disjoint chains on $P(n)$, all containing both barred and unbarred elements.
The number of rooks outside the lower-left quadrant comprising these chains
is $2n-|g|$; that is, $|g|$ rooks remain in the lower-left quadrant,
placed unrestrictedly in the unoccupied rows and columns.  
Reading these rooks as a permutation determines an order
on the chains of~$g$, and $g$ together with this order
gives an element of~$G(n)$.
\end{proof}

\begin{proof}[Proof of Proposition~\ref{prop:Genocchi}]
Using Lemma~\ref{lem:Genocchi}, what is left to show
is that the numbers $|G(n)|$ are the
median Genocchi numbers.
We will prove that they satisfy the Seidel triangle recurrence:
see for instance \cite[\S2]{EhrSte}.
Define the numbers $s_{k,i}$ for $k\geq 0$, $1\leq i\leq (k+3)/2$ by
$s_{0,1}=0$, and otherwise
\begin{align*}
s_{2k-1,i} &= s_{2k-2,i}+s_{2k-1,i-1},\\
s_{2k,i} &= s_{2k-1,i}+s_{2k,i+1},
\end{align*}
where summands $s$ with out-of-range indices should be interpreted as~0.
Then $s_{2n,1}$ is the $n$th median Genocchi number.
The first $s_{k,i}$ are tabulated below.
\begin{center}
\begin{tabular}{l|rrrrrrr}
$i$ \textbackslash\ $k$ & 0 & 1 & 2 & 3 & 4 & 5 & 6 \\\hline
1 & \textbf{1} & 1 & \textbf{2} & 2 & \textbf{8} & 8 & \textbf{56} \\
2 & & 1 & 1 & 3 & 6 & 14 & 48 \\
3 & & & & 3 & 3 & 17 & 34 \\
4 & & & & & & 17 & 17
\end{tabular}
\end{center}

In fact, we will show that the entry $s_{2n,i}$ of the Seidel triangle
is the number of elements of $G(n)$ such that the unbarred
elements $1, \ldots, i-1$ are each contained in a chain.  For instance, of the 8
elements of $G(2)$, 6 of them have a chain containing 1, and 3 of them
have both a chain containing 1 and a chain containing 2 (possibly the
same chain).  Equivalently, $s_{2n-1,i}$ will count the elements
of~$G(n)$ such that $1,\ldots,i-1$ are each contained in a chain,
but (if $i\leq n$) $i$ is not.

We exhibit a bijection between $G(n+1)$ and
\[\raisebox{-0.375\baselineskip}{\Big\{}
(g, i, j) : \begin{minipage}[t]{24em}
$g\in G(n), 1 \leq i \leq n+2, 1 \leq j \leq i$,  and in $g$\\
the elements $1, \ldots, j-1$ are each contained in a chain\end{minipage}
\raisebox{-0.375\baselineskip}{\Big\}},\]
such that if $g'$ is put in bijection with $(g, i, j)$,
then in $g'$ the elements $1, \ldots, i-1$ 
are contained in a chain, but $i$ is not.
(The condition on~$i$ is vacuous if $i = n+2$.)
This corresponds to the recurrence among the numbers $s_{2k,i}$:
indeed, $g'$ is one of the elements counted by $s_{2n+1,i}$, 
while $g$ is one of those counted by $s_{2n,j}$, so the bijection
realizes $s_{2n+1,i}$ as the partial sum of the $s_{2n,j}$ for $j\leq i$.

Given $g'$,  we construct $(g, i, j)$.  First, $i$ is determined by $g'$
as above: it's the least unbarred index not contained in a chain of
$g'$, or $n+2$ otherwise.

Let $h$ be obtained from $g'$ by deleting $i$ and $\bar i$ 
(or $n+1$ and $\overline{n+1}$ if $i = n+2$), and then decrementing
all labels greater than $i$, barred and unbarred, so that $h$ is
an ordered list of chains on~$P(n)$.
We break the construction of~$g$ and~$j$ into several cases for the purpose
of sketching the reverse bijection hereafter.

\begin{description}
\item[Case 1] If $i < n+2$ and $\bar i$ is contained in no chain in $g'$,
then we let $g = h$ and $j = i$.
\end{description}

Otherwise there is a chain $C$ in $g'$ which contains $\bar i$ 
(if $i < n+2$)  or $\overline{n+1}$ (if $i = n+2$).

\begin{description}
\item[Case 2] If $C$ has more than one barred element, then we let $g = h$
and $j$ be the greatest unbarred element of $C$.
\item[Case 3] If $C$ has only one barred element but is the last chain in the
ordered list $g'$,  then we let $g$ be $h$ with the chain that is the
remnant of $C$ deleted, and $j$  be the least (unbarred) element of $C$.
\end{description}

Otherwise, $C$ has only one barred element, and is followed in $g'$ by
another chain $D$.

\begin{description}
\item[Case 4] If the least element of $C$ is greater than the least element
of $D$, let $g$ be $h$ with the remnant of $C$ deleted, and $j$ be the
greatest unbarred element of $D$.
\item[Case 5] If the least element of $C$ is less than the least element of $D$
(call this $d$), then let $g$  be $h$ with all the elements of $C$ less
than $d$ inserted into the image of $D$ and the remainder of $C$ deleted,
and let $j$  be the greatest element of $C$ less than $d$.
\end{description}

We describe the inverse of this bijection more summarily. 
Let $(g,i,j)$ be in the image of the bijection.
Let $h'$ be the list of chains of $P(n+1)$
obtained from~$g$ by incrementing the labels in~$g$ 
that are greater than or equal to~$i$, whether barred or not.
Then $g'$ will equal either $h'$ or a small modification thereof.
If $j = i$ and $i<n+2$, we are in Case~1 above,
and $g'=h'$.  Otherwise, if
$j$ is not in a chain of $g$, including if $j=n+2$, we are in Case~3,
and $g'$ is obtained from $h'$ by appending the chain 
containing $\bar i$ and all unused unbarred labels strictly less than~$i$.
If $j$ is in a chain $D$ of
$g$ but is not the greatest unbarred element thereof, we are in Case~5,
and $g'$ is obtained from $h'$ by removing the unbarred elements
less than or equal to~$j$ from~$D$, and inserting before~$D$ a new chain
containing $\bar i$, these removed elements, 
and any unused unbarred labels strictly less than~$i$.
If $j$ is the greatest unbarred element of a chain $D$ of~$g$, we are in
Case~2 if every element $1, \ldots, i-1$ is in some chain of~$g$,
in which case $g'$ is obtained from $h'$ by inserting $\bar i$ into~$D$.
Otherwise we are in Case~4, and $g'$ is obtained from $h'$ by
inserting before~$D$ the chain 
containing $\bar i$ and all unused unbarred labels strictly less than~$i$.
\end{proof}




\subsection*{Acknowledgements}
We thank Allen Knutson for pointing out the connection 
between symmetric matrices and the symplectic Grassmannian,
and an anonymous referee for several valuable suggestions. In particular, Proposition \ref{prop:definedbyinclusions} and all connections to Gasharov and Reiner's work \cite{GasharovReiner} appears because of the referee's comments.  
Seth Sullivant was partially supported by the David and Lucille Packard 
Foundation and the US National Science Foundation (DMS 0954865). 

\bibliographystyle{amsalpha}      
\bibliography{SchubertsAndGaussiansBib}


%
%

\end{document}